\documentclass[11pt,reqno]{amsart}

\usepackage{amsmath,amsthm,amstext,amsfonts,amssymb,amsxtra}
\numberwithin{equation}{subsection}

\usepackage{mathtools}
\usepackage{IEEEtrantools}

\usepackage{tikz}
\usetikzlibrary{patterns}
\usetikzlibrary{decorations.pathreplacing, decorations.pathmorphing}
\usepackage{hyperref}
\usepackage{caption}
\usepackage{subcaption}

\usepackage{rotating}

\theoremstyle{definition}
\newtheorem{theorem}{Theorem}[subsection]									
\newtheorem{lemma}[theorem]{Lemma}      	 	
\newtheorem*{lemma*}{Lemma}           	
\newtheorem{proposition}[theorem]{Proposition}		
\newtheorem*{proposition*}{Proposition}    	
\newtheorem*{theorem*}{Theorem}  
\newtheorem*{defninition*}{Definiton}

\newtheorem{remark}[theorem]{Remark}  
\newtheorem*{remark*}{Remark}       

\title[Arithmeticity Ko-Zo monodromies II]{Arithmeticity of the Kontsevich--Zorich monodromies of certain families of square-tiled surfaces II}

\begin{document}

\bibliographystyle{plainnat}

\author{Manuel Kany}
\author{Carlos Matheus}

\address{Manuel, Kany, Department of Mathematics and Computer Science, Saarland University, 66123 Saarbr\"ucken, Germany}
\email{kany@math.uni-sb.de}

\address{Carlos, Matheus, Centre de Math\'ematiques Laurent Schwartz, CNRS (UMR 7640), \'Ecole Polytechnique, 91128 Palaiseau, France}
\email{carlos.matheus@math.cnrs.fr}

\maketitle

\begin{abstract} In this note, we extend the scope of our previous work joint with Bonnafoux, Kattler, Ni\~no, Sedano-Mendoza, Valdez and Weitze-Schmith\"usen by showing the arithmeticity of the Kontsevich--Zorich monodromies of infinite families of square-tiled surfaces of genera four, five and six.  
\end{abstract}

\tableofcontents

\section{Introduction}

Square-tiled surfaces (or origamis) are translation surfaces obtained from finite branched covers of the flat torus $\mathbb{R}^2/\mathbb{Z}^2$ which are unramified away from the origin. The $SL(2,\mathbb{R})$-orbits of square-tiled surfaces are closed subsets of the moduli spaces of translation surfaces called arithmetic Teichm\"uller curves. The first cohomology groups of the Riemann surfaces of genus $g\geq 2$ in an arithmetic Teichm\"uller curve form a variation of Hodge structures whose monodromy group is naturally isomorphic to a subgroup of $SL(2,\mathbb{Z})\times Sp(2g-2,\mathbb{Z})$. The projection of the monodromy group in the second factor of $SL(2,\mathbb{Z})\times Sp(2g-2,\mathbb{Z})$ is called the Kontsevich--Zorich monodromy of the square-tiled surface and, partly inspired by a question of Sarnak, one can try to decide if the Kontsevich--Zorich monodromy of a ``typical'' square-tiled surface is arithmetic or thin.\footnote{As it was remarked by Filip, the Kontsevich--Zorich monodromy of a ``typical'' square-tiled surface of genus $g$ is Zariski dense in $Sp(2g-2,\mathbb{R})$ (cf. Theorem 5.4.7 in Filip's survey  \cite{filip22}). In particular, a ``typical'' square-tiled surface of genus $g$ has arithmetic Kontsevich--Zorich monodromy when this group has finite index in $Sp(2g-2,\mathbb{Z})$.} 

M\"oller noticed that any square-tiled surface of genus two has an arithmetic Kontsevich--Zorich monodromy (cf. Appendix B of \cite{bonnafoux22}). Moreover, in our previous work \cite{bonnafoux22} joint with Bonnafoux, Kattler, Ni\~no, Sedano-Mendoza, Valdez and Weitze-Schmith\"usen, we showed that many square-tiled surfaces in the minimal stratum of the moduli space of translation surfaces of genus three have an arithmetic Kontsevich--Zorich monodromy. 

In this paper, we extend these results to certain families of square-tiled surfaces of genera four, five and six. We hereby use methods which are very similar to the methods used for genus three origamis in \cite{bonnafoux22} but we have to enrich our toolbox by more advanced concepts from Galois theory to find Galois pinching elements in the Kontsevich--Zorich monodromy for growing genera.

\begin{theorem}\label{t.A} Let $\mathcal{H}(2g-2)$ be the minimal stratum of the moduli space of translation surfaces of genus $g$. For each $g\in\{4,5,6\}$, there are infinitely many square-tiled surfaces in $\mathcal{H}(2g-2)$ whose Kontsevich--Zorich monodromies are arithmetic. 
\end{theorem} 

\begin{remark} As it turns out, we shall deduce this statement from the more precise results in Theorems \ref{t.A.g4}, \ref{t.A.g5} and \ref{t.A.g6} below. 
\end{remark} 

The remainder of this note concerns the proof of Theorem \ref{t.A}. More concretely, we organize this note as follows: in Section \ref{s.prelim}, we recall some Zariski-denseness and arithmeticity criteria used in \cite{bonnafoux22} together with a discussion of the Galois groups of reciprocal polynomials based on a result of Jackson \cite{jackson04} about the subgroups of the hyperoctahedral groups; afterwards, we complete the paper by showing the validity of Theorem \ref{t.A} in $\mathcal{H}(6)$, $\mathcal{H}(8)$ and $\mathcal{H}(10)$ (resp.) in Sections \ref{s.g4}, \ref{s.g5} and \ref{s.g6} (resp.). 

\begin{remark} In what follows, we shall assume some familiarity with the basic theory of square-tiled surfaces explained in \cite[\S 2]{bonnafoux22} (for instance).  
\end{remark} 

We close this short introduction with the following question: is it true that the Kontsevich--Zorich monodromy of a typical\footnote{In the sense of \cite{filip22}.} square-tiled surface is arithmetic? 

\textbf{Acknowledgments.}
The authors want to thank the referees for their comments and suggestions which vastly improved this article.

Furthermore Manuel Kany wants to thank the Deutsche Forschungsgemeinschaft (DFG, German Research Foundation) for their funding in Project-ID 286237555 – TRR 195.

\section{Preliminaries}\label{s.prelim}

\subsection{Square-tiled surfaces and their~Kontsevich--Zorich~monodromies}\label{ss.general-strategy} 

Recall that a square-tiled surface or origami is a translation surface $\mathcal{O}=(M,\omega)$, where the Riemann surface $M$ comes from a finite branched cover $\pi:M\to\mathbb{R}^2/\mathbb{Z}^2$ which is unramified away from $0\in\mathbb{R}^2/\mathbb{Z}^2$, and the Abelian differential $\omega$ is $\pi^*(dx+idy)$. The group $\textrm{Aff}(\mathcal{O})$ of affine homeomorphisms of an origami $\mathcal{O}$ respects the natural splitting 
\begin{align}\label{Eq:SplittingHomol}
  H_1(M,\mathbb{Q}) = H_1^{st}(\mathcal{O},\mathbb{Q})\oplus H_1^{(0)}(\mathcal{O},\mathbb{Q}),
\end{align}
where 
  \[
  H_1^{(0)}(\mathcal{O},\mathbb{Q})=\left\{\gamma\in H_1(M,\mathbb{Q}): \int_{\gamma}\omega = 0\right\}
  \]
and $H_1^{st}(\mathcal{O},\mathbb{Q})$ is the symplectic orthogonal of $H_1^{(0)}(\mathcal{O},\mathbb{Q})$ with respect to the usual intersection form on $H_1(M,\mathbb{Q})$. The decomposition in (\ref{Eq:SplittingHomol}) is for origamis defined over $\mathbb{Z}$. In this setting, the Kontsevich--Zorich monodromy of an origami $\mathcal{O}$ is the subgroup $\Gamma$ of $Sp(H_1^{(0)}(\mathcal{O},\mathbb{Z}))\simeq Sp(2g-2,\mathbb{Z})$ generated by the actions of the elements of $\textrm{Aff}(\mathcal{O})$ on $H_1^{(0)}(\mathcal{O},\mathbb{Z})$. 

The Zariski-denseness of the Kontsevich--Zorich monodromy $\Gamma$ of an origami $\mathcal{O}$ of genus $g$ can be checked in many contexts using the actions on homology of Dehn twists associated to cylinder decompositions in rational directions. More precisely, suppose that: 
\begin{itemize}
\item[(a)] 
we can combine such Dehn twists to produce a Galois pinching element $A\in Sp(H_1^{(0)}(\mathcal{O},\mathbb{Z}))$, i.e., a symplectic matrix whose characteristic polynomial is irreducible over $\mathbb{Z}$, splits over $\mathbb{R}$, and possesses the largest possible Galois group among reciprocal polynomials of degree $2g-2$, the hyperoctahedral group; 
\item[(b)] 
some Dehn twist induces a non-trivial unipotent element 
  \[
  B\in Sp(H_1^{(0)}(\mathcal{O},\mathbb{Z}))
  \]
such that $(B-\textrm{Id})(H_1^{(0)}(\mathcal{O},\mathbb{R}))$ is not a Lagrangian subspace of $H_1^{(0)}(\mathcal{O},\mathbb{R})$.  
\end{itemize} 
Then, $\Gamma$ is Zariski-dense in $Sp(H_1^{(0)}(\mathcal{O},\mathbb{R}))$ thanks to \cite[Theorem 9.10]{prasad14} and \cite[Proposition 4.3]{matheus15} . 

Furthermore, if $\mathcal{O}$ is an origami whose Kontsevich--Zorich monodromy $\Gamma$ is Zariski-dense in $Sp(H_1^{(0)}(\mathcal{O},\mathbb{R}))$, then Singh--Venkataramana \cite{singh14} showed that $\Gamma$ is arithmetic (i.e., it has finite index in $Sp(H_1^{(0)}(\mathcal{O},\mathbb{Z}))$) provided that it contains three unipotent matrices $T_n$, $n=1,2,3$, with 
  \[
  (T_n-\textrm{Id})(H_1^{(0)}(\mathcal{O},\mathbb{Z}))=\mathbb{Z}\,w_n,
  \]
such that $W=\mathbb{Q}w_1\oplus \mathbb{Q}w_2\oplus \mathbb{Q}w_3$ is not an isotropic subspace and the group $\langle T_n|_W:n=1,2,3\rangle$ contains a non-trivial element of the unipotent radical of $Sp(W)$. 

The Zariski-denseness and arithmeticity criteria above were used in \cite{bonnafoux22} to establish the abundance of arithmetic Kontsevich--Zorich monodromies among origamis in $\mathcal{H}(4)$. In order to extend this kind of result to minimal strata $\mathcal{H}(6)$, $\mathcal{H}(8)$ and $\mathcal{H}(10)$, we shall need the Galois-theoretical facts described in the next two subsections. 

\subsection{Galois groups as permutation group}\label{SubSec:GalGasPG}
Consider a monic irreducible  polynomial $P(X)\in\mathbb{Z}[X]$ of degree $n$ with set of complex roots $S=\{\lambda_1,\dots,\lambda_n\}$. Let $Z(P)=\mathbb{Q}(\lambda_1,\dots,\lambda_n)$ be the splitting field of the polynomial $P(X)\in\mathbb{Z}[X]$. We consider the standard embedding of $\text{Gal}(P)=\text{Aut}_\mathbb{Q}(Z(P))$ in the permutation group $\text{Sym}(S)$ via
  \[
  \text{Gal}(P)\longrightarrow \text{Sym}(S),\quad \sigma\longmapsto \sigma|_S.
  \] 
The theorem of Dedekind is a useful tool to study the Galois group of a polynomial $P(X)\in\mathbb{Z}[X]$ as above:

\begin{theorem}[Dedekind]
Let $P(X)\in\mathbb{Z}[X]$ be monic irreducible of degree $n$. For every prime number $p$ not dividing the discriminant of $P(X)\in\mathbb{Z}[X]$, let the monic irreducible factorization of $P(X)\in\mathbb{Z}[X]$ modulo $p$ be
  \[
  P(X) \equiv \pi_1(X)\cdots\pi_k(X) ~\text{mod}~p
  \]
with $\pi_i(X)$ pairwise distinct and set $d_i:=\deg{\pi_i(X)}$, so $d_1+\dots+d_k=n$. Then the Galois group $\text{Gal}(P)$ of $P(X)\in\mathbb{Z}[X]$ viewed as a subgroup of $\text{Sym}(S)$ contains an element that permutes the roots $S$ of $P(X)$ with cycle type $(d_1,\dots,d_k)$.
\end{theorem}

\subsection{Galois groups of polynomials of degree four and five}\label{SubSec:GalG45}

For a reference of the following see \cite{jensen02}. 
We consider in this subsection a monic irreducible polynomial 
  \[
  Q(X)=X^k+\sum_{i=0}^{k-1} b_i\,X^i\in \mathbb{Q}[X]
  \]
of degree four or five with set of roots $S=\{\mu_1,\dots,\mu_k\}$ ($k=4$ or $k=5$).

Let first $k=4$. We define the \textit{cubic resolvent} $CR_Q(Y)\in\mathbb{Q}[Y]$ of the polynomial $Q(X)$ as 
\begin{align*}
  CR_Q(Y)=(Y-(\mu_1\mu_2+\mu_3\mu_4))\,(Y-(\mu_1\mu_3+\mu_2\mu_4))\,(Y-(\mu_1\mu_4+\mu_2\mu_3)).
\end{align*}
Direct calculations show 
\begin{displaymath}
    CR_Q(Y)=Y^3-b_2\,Y^2+(b_1\,b_3-4b_0)\,Y-(b_0\,b_3^3-4\,b_0\,b_2+b_1^2).
\end{displaymath} 
Furthermore a computation reveals the equality 
 \[
 \text{Disc}(Q(X))=\text{Disc}(CR_Q(Y))
 \]
between the discriminant $\text{Disc}(Q(X))$ of $Q(X)\in\mathbb{Q}[X]$ and the discriminant $\text{Disc}(CR_Q)$ of $CR_Q(Y)\in\mathbb{Q}[Y]$.
The five transitive subgroups of the permutation group $S_4$ are the Klein-four group $V_4$, $C_4$, the dihedral group $D_4$, $A_4$ and $S_4$ itself. The next theorem will help to determine the Galois group of the polynomial $Q(X)\in\mathbb{Q}[X]$. A proof can be found in \cite[Theorem 2.2.2]{jensen02}.

\begin{theorem}\label{Thm:CRGalGr}
Let $Z(CR_Q)$ be the splitting field of the cubic resolvent $CR_Q(Y)\in\mathbb{Q}[Y]$ from above and let $m=[Z(CR_Q):\mathbb{Q}]$ be the degree of the field extension over the rational numbers. Then we have for the Galois group $\text{Gal}(Q)\leq S_4$ of the irreducible polynomial $Q(X)\in\mathbb{Q}[X]$ from above:
\begin{displaymath}
\text{Gal}(Q)=
\begin{cases}
S_4               & \mbox{if $m=6$} \\
A_4               & \mbox{if $m=3$} \\
D_4~\text{or}~C_4 & \mbox{if $m=2$} \\
V_4               & \mbox{if $m=1$}
\end{cases}
\end{displaymath}
\end{theorem}

\begin{remark}\label{Rem:CRandDisQ}
Since the cubic resolvent $CR_Q(Y)\in\mathbb{Q}(Y)$ of $Q(X)\in\mathbb{Q}[X]$ is a degree three polynomial it is sufficient for the splitting field $Z(CR_Q)$ to be a degree six field extension over the rational numbers, that $CR_Q(Y)\in\mathbb{Q}(Y)$ is irreducible and that the discriminant $\text{Disc}(CR_Q)=\text{Disc}(Q)$ of $CR_Q(Y)$ respectively $Q(X)$ is not a square of a rational number.
\end{remark}

Now let $k=5$. The \textit{Weber sextic resolvent} $SWR_Q(Y)\in\mathbb{Q}[Y]$ defined as in Definition 2.3.2 of \cite{jensen02} helps to determine the Galois group of the quintic polynomial $Q(X)\in\mathbb{Q}[X]$ as we will explain in the following Theorem and Remark (see \cite[Theorem 2.3.3]{jensen02} for a proof of the theorem):

\begin{theorem}\label{Thm:SexticWR}
The Galois group $\text{Gal}(Q)\leq S_5$ of the irreducible monic polynomial $Q(X)\in\mathbb{Q}[X]$ is solvable if and only if the sextic Weber resolvent $SWR_Q(Y)\in\mathbb{Q}[Y]$ has a root in the rational numbers $\mathbb{Q}$.
\end{theorem}

\begin{remark}\label{Rem:SWRFullGG}
The only transitive subgroups of $S_5$ are $C_5$, the dihedral group $D_5$, $F_{20}$, $A_5$ and $S_5$ itself. Hence the only transitive subgroups that are non-solvable are $A_5$ and $S_5$. Thus it is easy to ensure that $\text{Gal}(Q)=S_5$ with the help of the previous theorem and the fact that $\text{Gal}(Q)\leq A_5$ if and only if the discriminant $\text{Disc}(Q)$ is a square of a rational number.
\end{remark}

\subsection{Galois group of certain reciprocal polynomials }\label{SubSec:RecHyp}
Consider from now on an irreducible monic polynomial 
  \[
  P(X)=\sum_{i=0}^n c_i\,X^i\in\mathbb{Z}[X]
  \]
of degree $n=2k$ with $c_n=c_0=1$ and $c_i=c_{n-i}$ for $i=1,\dots,k$. 

The Galois group of such a $P(X)\in\mathbb{Z}[X]$ can be naturally seen as a subgroup of the hyperoctahedral group $G_k$ as we will see in Section \ref{SubSec:DescGalGr}. But first we want to explain the hyperoctahedral group as we want to use it in this text:

For $k\geq 1$ we define the hyperoctahedral group as the semidirect product $G_k=\mathbb{Z}_2^k\rtimes S_k$, where $S_k$ is the permutation group on a set with $k$ elements. 
The group $S_k$ acts on $\mathbb{Z}_2^k$ by $\tau(\epsilon_1,\dots,\epsilon_k)=(\epsilon_{\tau(1)},\dots,\epsilon_{\tau(k)})$ $(\epsilon_i\in\{\pm1\})$ and the multiplication on the hyperoctahedral group $G_k$ is defined by
  \begin{align}\label{Eq:ProdHypGr}
  (\epsilon,\tau)\cdot (\tilde\epsilon,\tilde \tau)=({\tilde\tau}^{-1}(\epsilon) \cdot \tilde\epsilon,~\tau\circ\tilde\tau).
  \end{align}
Compare section two in \cite{jackson04}.

\subsubsection{Description of the Galois group}\label{SubSec:DescGalGr}

Let again $P(X)$ as in the beginning of the section. Since $P(X)\in\mathbb{Z}[X]$ is reciprocal and its splitting field is of zero characteristic and hence perfect, the polynomial $P(X)$ has $n=2k$ distinct roots which come in pairs $\{\lambda_i,\, \lambda_i^{-1}\}$ for $i\in\{1,\dots,k\}$. Denote by $Z(P)$ the splitting field of the polynomial $P(X)\in\mathbb{Z}[X]$ and by $\text{Gal}(P)=\text{Aut}_\mathbb{Q}(Z(P))$ the Galois group of $P(X)\in\mathbb{Z}[X]$. An automorphism $\sigma\in\text{Gal}(P)$ necessarily permutes the $k$ pairs of roots $\{\lambda_i,\,\lambda_i^{-1}\}$ of $P(X)\in\mathbb{Z}[X]$ and this leads to a group homomorphism
  \[
  \phi\colon \text{Gal}(P)\longrightarrow S_k,\quad \sigma\longmapsto \tau_\sigma,
  \]
where we set $\tau_\sigma(i)=j$ if $\sigma(\{\lambda_i,\,\lambda_i^{-1}\})=\{\lambda_j,\,\lambda_j^{-1}\}$ $(i,j\in\{1,\dots,k\})$. The kernel $N$ of $\phi$ is given by the automorphisms $\sigma\in \text{Gal}(P)$ such that $\sigma(\{\lambda_i,\,\lambda_i^{-1}\})=\{\lambda_i,\,\lambda_i^{-1}\}$ for every $i\in\{1,\dots,k\}$. Hence, we can identify $N$ with a subgroup of $\mathbb{Z}_2^k$ via the following homomorphism
  \[
  \iota\colon N \longrightarrow \mathbb{Z}_2^k,\quad 
  \sigma\longmapsto \epsilon^\sigma=(\epsilon_1^\sigma,\dots,\epsilon_k^\sigma),
  \]
where $\epsilon_i^\sigma=1$ if $\sigma(\lambda_i)=\lambda_i$ and $\epsilon_i^\sigma=-1$ if $\sigma(\lambda_i)=\lambda_i^{-1}$. Since $N$ is the kernel of $\phi$, it is normal and together with the map $\iota$, we can represent $\text{Gal}(P)=N\rtimes \text{Im}(\phi)$ as a subgroup of the hyperoctahedral group $G_k=\mathbb{Z}_2^k\rtimes S_k$. The whole situation is visualized in the following commutative diagram:
\begin{center}
\begin{tikzpicture}[scale=1.5]
\draw(-1,0) node[left]{$N$};
\draw[->](-1,0) -- node[above]{$\iota$} (0.7,0) node[right]{$\text{Gal}(P)$};
\draw[->](1.6,0) -- node[above]{$\phi$} (3,0) node[right]{$S_k$};
\draw[->](3.2,-0.2) -- (3.2,-0.7) node[below]{$S_k$};
\draw[->](1.2,-0.2) -- (1.2,-0.7) node[below]{$G_k=\mathbb{Z}_2^k\rtimes S_k$};
\draw[->](-1.2,-0.2) -- (-1.2,-0.7) node[below]{$\mathbb{Z}_2^k$};
\draw[->](-1,-0.9) --node[above]{$i_k$} (0.3,-0.9);
\draw[->](2.1,-0.9) -- node[above]{$\pi_k$} (3,-0.9);
\end{tikzpicture}
\end{center}

Here $i_k\colon \mathbb{Z}_2^k\to G_k$ is the inclusion map and $\pi_k\colon G_k\to S_k$ is the projection map, what makes of course $\mathbb{Z}_2^k\overset{i_k}{\longrightarrow} G_k\overset{\pi_k}{\longrightarrow}S_k$ a split exact sequence. 

\subsubsection{Action of the hyperoctahedral group on the splitting field}\label{Rem:ActHOGroup}
In the following we also want to consider the action of the hyperoctahedral group $G_k$ on the splitting field $Z(P)$ of $P(X)\in\mathbb{Z}[X]$ that fits together with the action of $\text{Gal}(P)$ from above. Hence we define the action of $G_k$ on $Z(P)$ via a permutation of the roots $\{\lambda_i,\,\lambda_i^{-1}\mid i=1,\dots,k\}$ in the following way: For $\epsilon=(\epsilon_1,\dots,\epsilon_k)\in\{(\pm 1,\dots, \pm 1)\}$, every $\tau\in S_k$  and $i=1,\dots,k$ we define
\begin{align}\label{Eq:GrActField}
(\epsilon,\tau)\,.\,\lambda_i = \lambda_{\tau(i)}^{\epsilon_i} \quad \text{and} 
\quad (\epsilon,\tau)\,.\,\lambda_{i}^{-1} = \lambda_{\tau(i)}^{-\epsilon_i}.
\end{align}

\begin{lemma}
    The equalities in (\ref{Eq:GrActField}) indeed define an action of the group $G_k$ on the set of roots $\{\lambda_i,\,\lambda_i^{-1}\mid i=1,\dots,k\}$.
\end{lemma}
\begin{proof}
We have to show the compatibility with the multiplication on $G_k$ defined in $(\ref{Eq:ProdHypGr})$. Let $\delta\in\{\pm 1\}$ and 
$i\in\{1,\dots,k\}$. For $(\epsilon,\tau), (\tilde\epsilon,\tilde \tau)\in G_k$ we have
\begin{align*}
     (\epsilon,\tau).\Bigl( (\tilde\epsilon,\tilde \tau)\,.\,\lambda_i^\delta \Bigr) &=(\epsilon,\tau)\,.\,\lambda_{\tilde\tau(i)}^{\delta\cdot\tilde\epsilon_i}
   = \lambda_{\tau\circ\tilde\tau(i)}^{(\delta\cdot \tilde\epsilon_i)\cdot \epsilon_{\tilde\tau(i)}}
   =(\tilde\tau^{-1}(\epsilon)\cdot\tilde\epsilon,~\tau\circ\tilde\tau)\,.\,\lambda_i^\delta ,
\end{align*}
what ends the proof since $(\epsilon,\tau)\cdot (\tilde\epsilon,\tilde \tau)=({\tilde\tau}^{-1}(\epsilon) \cdot \tilde\epsilon,~\tau\circ\tilde\tau)$ by definition of the product on $G_k$.
\end{proof}

In our subsequent discussion, we shall need to describe the Galois groups $\text{Gal}(P)$ such that the map $\phi\colon\text{Gal}(P)\to S_k$ is onto. The next Proposition will help us with that.

\begin{proposition}\label{Prop:SubgrPhiOnto}
Let $k\geq 2$ and $P(X)\in\mathbb{Z}[X]$ be an irreducible reciprocal polynomial of degree $2k$ with $\phi(\text{Gal}(P))=S_k$. Then $\text{Gal}(P)$ is isomorphic to one of the following subgroups of $G_k=\mathbb{Z}_2^k\rtimes S_k$:

Either $\text{Gal}(P)\simeq S_k$, $\text{Gal}(P) \simeq G_k$ or $\text{Gal}(P)$ is isomorphic to one of the following three subgroups $H_{k,1},\, H_{k,2}, \, H_{k,3}\leq G_k$, where
\begin{align*}
H_{k,1}& :=\{((\epsilon_1,\dots,\epsilon_k),\tau)\mid ~\prod_{i=1}^k \epsilon_i = 1\},\\
H_{k,2}& := \{((\epsilon_1,\dots,\epsilon_k),\tau)\mid ~\textrm{sign}(\tau)\, \prod_{i=1}^k\epsilon_i= 1\},\\
\text{and}\quad
H_{k,3}& :=\{(+1,\dots,+1),\,(-1,\dots,-1)\}\times S_k.
\end{align*}
\end{proposition} 

\begin{proof}
For $k=2$ this is an easy exercise since the only possible subgroups of $G_2=\mathbb{Z}_2^2\rtimes S_2$ which surject onto $S_2$ and which are not the full group, are groups of order four. In this case 
\begin{displaymath}
H_{2,1}=H_{2,3}=\langle ~(\lambda_1,\lambda_2)(\lambda_1^{-1},\lambda_2^{-1}) , (\lambda_1,\lambda_1^{-1})(\lambda_2,\lambda_2^{-1})~\rangle \simeq V_4 
\end{displaymath}
and
\begin{displaymath}
H_{2,2}=\langle ~(\lambda_1\lambda_2^{-1}\lambda_1^{-1}\lambda_2)~ \rangle\simeq C_4.
\end{displaymath}
For $k=3,4$ and all $k\geq 5$ the result follows as in Proposition 4 in \cite{jackson04} since $A_3$ is a simple group as well as all $A_k$ with $k\geq 5$ and the only non-trivial normal subgroup of $A_4$ is the Klein four-group $V_4$ which has index three in $A_4$.
\end{proof}

Together with the fundamental theorem of Galois theory we will use the next lemma to ensure that the Galois group of certain reciprocal polynomials $P(X)\in\mathbb{Z}[X]$ equals the whole hyperoctahedral group, i.e. $\text{Gal}(P)=G_k$.

\begin{lemma}\label{Lem:MaxSubDiscr}
Let the hyperoctahedral group $G_k=\mathbb{Z}_2^k\rtimes S_k$ act on the splitting field $Z(P)=\mathbb{Q}(\{\lambda_i,\, \lambda_i^{-1}\mid i=1,\dots,k\})$ of the polynomial $P(X)\in\mathbb{Z}[X]$ as explained in Section \ref{Rem:ActHOGroup}. Consider the two subgroups $H_{k,1}$ and $H_{k,2}$ of $G_k$ defined in Proposition \ref{Prop:SubgrPhiOnto}.
Then:
\begin{enumerate}
\item[(i)]
The expression $\delta_{k,1}:=\prod_i\Bigl(\lambda_i-\lambda_i^{-1}\Bigr)$ is invariant under the action of $H_{k,1}$ but not $G_k$.
\item[(ii)]
The expression 
$\delta_{k,2}:= \prod_{i<j}\Bigl(\lambda_i+\lambda_i^{-1}-\lambda_j-\lambda_j^{-1}\Bigr)
\prod_i \Bigl(\lambda_i-\lambda_i^{-1}\Bigr)$
is invariant under the action of $H_{k,2}$ but not $G_k$.
\end{enumerate}
\end{lemma}

\begin{proof}
For every $\tau\in S_k$ and every $i\in\{1,\dots,k\}$ we have 
  \[
  ((+1,\dots,+1),\tau)\,.\,(\lambda_i-\lambda_i^{-1})=\lambda_{\tau(i)}-\lambda_{\tau(i)}^{-1}.
  \]
This shows that $\delta_{k,1}\in Z(P)$ is invariant under the action of $S_k$.

Now consider  $(\epsilon,\text{id}).(\lambda_i-\lambda_i^{-1})$ for $\epsilon=(\epsilon_1,\dots,\epsilon_k)\in \mathbb{Z}_2^k$ and $i\in\{1,\dots,k\}$. We have 
\begin{align*}
  (\epsilon,\text{id})\,.\,(\lambda_i-\lambda_i^{-1})=\lambda_i-\lambda_i^{-1} = \epsilon_i\,(\lambda_i-\lambda_i^{-1})
  \quad &
  \mbox{if $\epsilon_i = 1$},\\
  (\epsilon,\text{id})\,.\,(\lambda_i-\lambda_i^{-1})=\lambda_i^{-1}-\lambda_i= \epsilon_i\,(\lambda_i-\lambda_i^{-1})
  \quad &
  \mbox{if $\epsilon_i = -1$}.
  \end{align*}
Furthermore $(\epsilon,\text{id}).(\lambda_i+\lambda_i^{-1})=\lambda_i+\lambda_i^{-1}$ for every $\epsilon\in\mathbb{Z}_2^k$. 

We know $\text{sgn}(\tau)=(-1)^{\text{inv}(\tau)}$ for every element $\tau\in S_k$ with $\text{inv}(\tau)\in\mathbb{N}_0$ is the number of elements $(i,j)\in\{1,\dots,k\}\times\{1,\dots,k\}$ with $i<j$ but $\tau(i)>\tau(j)$ and thus
\begin{align*}
    & ((+1,\dots,+1),\tau)\,.\,\Bigl(\prod_{i<j} (\lambda_i+\lambda_i^{-1}-\lambda_j-\lambda_j^{-1})\Bigr) \\
  = & ~\text{sgn}(\tau)\, \Bigl(\prod_{i<j} (\lambda_i+\lambda_i^{-1}-\lambda_j-\lambda_j^{-1})\Bigr).
\end{align*}
Putting this together with the arguments from above we conclude for every $(\epsilon,\tau)\in G_k=\mathbb{Z}_2^k\rtimes S_k$ where $\epsilon=(\epsilon_1,\dots,\epsilon_k)\in\{(\pm 1,\dots,\pm 1)\}$:
\begin{displaymath}
(\epsilon,\tau)\,.\,\delta_{k,1}=\left( \prod_{i=1}^k \epsilon_i \right) \, \delta_{k,1} 
\quad \text{and} \quad
(\epsilon,\tau)\,.\,\delta_{k,2}=\text{sgn}(\tau)\, \left( \prod_{i=1}^k \epsilon_i \right) \,  \delta_{k,2} 
\end{displaymath}
This proves (i) and (ii). 
\end{proof}

\begin{remark}\label{Rem:SqDelta}
For a reciprocal monic polynomial $P(X)\in\mathbb{Z}[X]$ of degree $n=2k$ with set of roots $\{\lambda_i,\,\lambda_i^{-1}\mid i=1,\dots,k\}\subset \mathbb{C}$ we can always find a polynomial $Q(Y)\in\mathbb{Z}[Y]$ such that $1/X^k\cdot P(X)=Q(X+1/X+2)$. The polynomial $Q(Y)\in\mathbb{Z}[Y]$ has distinct roots $\mu_i$ ($i=1,\dots,k$) such that without loss of generality $\mu_i=\lambda_i+\lambda_i^{-1}+2$ for all $i=1,\dots,k$. We write in the following $\Delta_{k,1}:=\delta_{k,1}^2$ and $\Delta_{k,2}:=\delta_{k,2}^2$ for the squares of the expressions $\delta_{k,1}$ and $\delta_{k,2}$ from Lemma \ref{Lem:MaxSubDiscr}. We have
\begin{enumerate}
\item[(i)]
$\Delta_{k,1} = \delta_{k,1}^2
              = \prod_{i=1}^k (\lambda_i-\lambda_i^{-1})^2=\prod_{i=1}^k\mu_i(\mu_i-4) = Q(0)\, Q(4)$
\end{enumerate}
and
\begin{enumerate}
\item[(ii)]
$\Delta_{k,2} = \delta_{k,2}^2 = \Bigl(\prod_{i<j} (\mu_i-\mu_j)\Bigr)^2       
                  \Bigl(\prod_i (\lambda_i-\lambda_i^{-1})\Bigr)^2 
               = \text{Disc}(Q)\, \Delta_{k,1}.$
\end{enumerate}
This shows $\Delta_{k,1},~\Delta_{k,2}\in\mathbb{Q}$ and delivers an easy way how to write $\Delta_{k,1}$ and $\Delta_{k,2}$ in terms of coefficients of $Q(Y)\in\mathbb{Q}[Y]$.
\end{remark}

\subsection{Real roots of cubic, quartic and quintic polynomials}\label{Sec:RealRoots}

\subsubsection{Real roots for cubic polynomials}

For a cubic polynomial $Q(X)\in\mathbb{R}[X]$ with discriminant $\text{Disc}(Q)\neq 0$ it is well known that the number of real roots can be read from the discriminant as follows. If
\begin{align}
\begin{split}
    \text{Disc}(Q)>0, &\quad \mbox{then $Q(X)$ has three real roots},\\
    \text{Disc}(Q)<0, &\quad \mbox{then $Q(X)$ has one real root and two non-real roots}.
\end{split}
\end{align}

\subsubsection{Real roots for quartic polynomials}
We have a slightly more complicated statement of this form for quartic polynomials as well, so let 
 \[
 Q(X)=X^4+a\,X^3+b\,X^2+c\,X+d\in\mathbb{R}[X]
 \]
be a real monic polynomial of degree four. By substituting $X=Y-a/4$, we get the depressed quartic polynomial
\begin{equation}\label{Eq:PolyDep}
  DQ(Y)=Y^4+q\,Y^2+rY+s\in\mathbb{R}[Y],
\end{equation}
with coefficients 
\begin{IEEEeqnarray*}{l}
  q = b-(3/8)\,a^2, \quad r= c-(1/2)\,a\,(b-(1/4)\,a^2) \quad\text{and}\\
  s = d-(3/256)\,a^4+(1/16)\,a^2\,b-(1/4)\,a\,c.  
\end{IEEEeqnarray*}

For a quartic polynomial in the depressed form as in (\ref{Eq:PolyDep}), there is an easy criterion whether the polynomial has four real roots or no real roots (see \cite{garver33}). We want to state the result and denote by $\text{Disc}(DQ)$ the discriminant of the polynomial in \ref{Eq:PolyDep} and by $F(DQ)$ the expression $F(DQ):=q^2-4s$. If we have
\begin{align}\label{Eq:DisRealRoots}
\begin{split}
    \text{Disc}(DQ)>0,~q\leq 0,     &\quad \mbox{then (\ref{Eq:PolyDep}) has no real roots},\\
    \text{Disc}(DQ)>0,~F(DQ)\leq 0, &\quad \mbox{then (\ref{Eq:PolyDep}) has no real roots},\\
    \text{Disc}(DQ)>0,~q<0,~F(DQ)<0,&\quad \mbox{then (\ref{Eq:PolyDep}) has four real roots}.\\
\end{split}
\end{align}

\subsubsection{Real roots for quintic polynomials}
We want to end this section with a criterion from \cite{Hou96} with which we can find out whether a quintic polynomial in $\mathbb{R}[X]$ has simple real roots. Let
  \[
  Q(X)=X^5+a\,X^4+b\,X^3+c\,X^2+d\,X+e\in\mathbb{R}[X].
  \]
By substituting $X=Y-a/5$ we get a depressed polynomial
\begin{align}\label{Eq:PolyDep5}
  DQ(Y)=Y^5+p\,Y^3+q\,Y^2+r\,Y+s\in\mathbb{R}[X].
\end{align}
With the help of the following four discriminants we can find out wether $DQ(Y)$ has simple real roots. We define
\begin{equation}
  \begin{split}
    \begin{IEEEeqnarraybox}[][c]{lCl} \label{Eq:Dis5RealRoots}
    F_1(DQ) & = & -p,\\
    F_2(DQ) & = & 40\,r\,p-12\,p^3-45\,q^2,\\
    F_3(DQ) & = & 12\,p^4\,r-4\,p^3\,q^2+117\,p\,r\,q^2-88\,r^2\,p^2-40\,q\,p^2\,s \\
            &   &  +125\,p\,s^2 -27\,q^4-300\,q\,r\,s+160\,r^3,\\
    F_4(DQ) & = & -1600\,q\,s\,r^3-3750\,p\,s^3\,q+2000\,p\,s^2\,r^2 -4\,p^3\,q^2\,r^2 \\     
            &   & +16\,p^3\,q^3\,s-900\,r\,s^2\,p^3 +825\,q^2\,p^2\,s^2 +144\,p\,q^2\,r^3  \\
            &   & + 2250\,q^2\,r\,s^2+16\,p^4\,r^3+108\,p^5\,s^2 -128\,r^4\,p^2-27\,q^4\,r^2 \\
            &   & +108\,q^5\,s +256\,r^5+3125\,s^4-72\,p^4\,r\,s\,q +560\,r^2\,p^2\,s\,q \\
            &   & -630\,p\,r\,q^3\,s.
    \end{IEEEeqnarraybox}
  \end{split}
\end{equation} 
In \cite{Hou96} they classified the number of real roots and their multiplicity of a depressed quintic polynomial as in (\ref{Eq:PolyDep5}) using six discriminants among which are the four discriminants from Equation (\ref{Eq:Dis5RealRoots}). We only state the for us relevant case, namely if the four discriminants $F_1(DQ)$, $F_2(DQ)$, $F_3(DQ)$ and $F_4(DQ)$ are positive then the polynomial $DQ(Y)$ in (\ref{Eq:PolyDep5}) has five simple real roots.


\section{Genus four stairs}\label{s.g4}

We found our infinite families of origamis with arithmetic Kontsevich--Zorich monodromies among ``stairs origamis'' as shown in Figure \ref{fig:OrigamiONM4}, Figure \ref{fig:OrigamiNM5} and Figure \ref{fig:OrigamiNM6}.
We used cylinder decomposition in certain directions to construct elements in the Kontsevich--Zorich monodromies of the mentioned origamis from above. Hereby we started analysing small cases of the ``stairs prototypes'' and then tried to extend them by adding a number of squares satisfying adequate arithmetic conditions to keep the same features of the cylinder decomposition. Also, we think that this kind of ``uniform'' approach based on a few directions could be generalised once $SL(2,\mathbb{Z})$-orbits of origamis are classified.

\subsection{Dehn twists in genus four}

\begin{figure}
\centering
   \begin{tikzpicture}[scale=0.8]
	\draw (0,0) rectangle (3,1);
	\draw (5,0) rectangle (6,1);
	\draw (0,1) rectangle (3,2);
	\draw (0,2) rectangle (2,3);
	\draw (0,3) rectangle (1,4);
	\draw (0,0) rectangle (1,4);
	\draw (0,6) rectangle (1,8);
	\draw (1,0) -- (1,3);
	\draw (2,0) -- (2,2);
	\draw (3,0) -- (3,1);
	\draw (5,0) -- (5,1);
	\draw (0,6) -- (1,6);
	\draw (0,7) -- (1,7);
    \draw [dashed] (3.8,0.5) -- (4.2,0.5);
    \draw [dashed] (0.5,4.8) -- (0.5,5.2);
    \draw [decorate,line width=0.5mm,decoration={brace}] (-0.5,4) --  (-0.5,8) node[pos=0.5,left=10pt,black]{$2m$};
    \draw [decorate,line width=0.5mm,decoration={brace,mirror}] (0,-1.1) --  (6,-1.1) node[pos=0.5,below=10pt,black]{$N$};
    \draw [decorate, decoration={snake}] (0,4.5) -- (1,4.5);
    \draw [decorate, decoration={snake}] (0,5.5) -- (1,5.5);
    \draw (0,4) -- (0,4.5);
    \draw (1,4) -- (1,4.5);
    \draw (0,5.5) -- (0,6);
    \draw (1,5.5) -- (1,6);    
    \draw [decorate, decoration={snake}] (3.5,0) -- (3.5,1);
    \draw [decorate, decoration={snake}] (4.5,0) -- (4.5,1);
    \draw (3,0) -- (3.5,0);
    \draw (3,1) -- (3.5,1);
    \draw (4.5,0) -- (5,0);
    \draw (4.5,1) -- (5,1);  
\end{tikzpicture}
    \caption{The origami $\mathcal{O}_{M,M}^{(4)}$.}
    \label{fig:OrigamiONM4}
\end{figure}
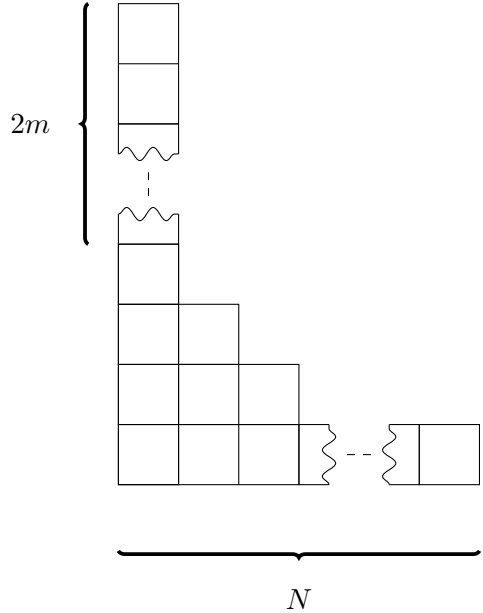

Let \(N\geq 4\) and \(M=4+2m\) with \(m\geq 0\). We consider the origami \(\mathcal{O}_{N,M}^{(4)}\) associated to the pair of permutations \(h,v \in\mbox{Sym}(\{1,\dots,N+M+2\})\), where
\begin{align*}
h=&(1,2,3\dots,N)(N+1,N+2,N+3)(N+4,N+5)(N+6)\dots(M)\\[2mm]
v=&(1,\ N+1,\ N+4,\ N+6,\dots, N+M)(2,\ N+2,\ N+5)(3,\ N+3)\\
  &(4)\dots(N).
\end{align*}
see Figure \ref{fig:OrigamiONM4}.

The Kontsevich--Zorich monodromy of \(\mathcal{O}_{N,M}^{(4)}\) was studied in Section five of \cite{bonnafoux22} via the analysis of Dehn twists in several rational directions. You can find all the details how we constructed the following matrices in Section five of the article \cite{bonnafoux22} and so we will omit this part. 

For the purpose of finding a Galois pinching element in $Sp(H_1^{(0)}(\mathcal{O}_{N,M}^{(4)},\mathbb{Z}))$, we consider the horizontal and vertical directions which lead to Dehn twists acting on an appropriate basis of $H_1^{(0)}(\mathcal{O}_{N,M}^{(4)},\mathbb{Q})$ via the matrices 
\begin{align*}
  M_h^{(0)}=&
    \left(\begin{array}{cccccc}
    1& 0& 0&   0& 3N& 3N\\
    0& 1& 0& 2N& 2N& 2N\\
    0& 0& 1& -6& -12& -6(M-1)\\
    0& 0& 0&  1&   0& 0 \\
    0& 0& 0& 0 & 1  & 0 \\
    0& 0& 0& 0 & 0  & 1 
    \end{array}\right),\\[2mm]
  M_v^{(0)}=&
    \left(\begin{array}{cccccc}
    1   &     0&       0& 0& 0& 0\\
    0   &     1&       0& 0& 0& 0\\
    0   &     0&       1& 0& 0& 0\\
    0   & -3M&    -3M& 1& 0& 0\\
    -2M& -2M&   -2M& 0& 1& 0\\
     6   &12  & 6(N-1)& 0& 0& 1
    \end{array}\right).
\end{align*}

Furthermore, in \cite[Section 5.3]{bonnafoux22} the authors showed that the directions \((1,1)\), \((1,-1)\) and \((1,2)\) provide Dehn twists acting on a non-isotropic three dimensional subspace $W\subset H_1^{(0)}(\mathcal{O}_{N,M}^{(4)},\mathbb{Q})$ via the matrices 
\begin{align*}
  D_\delta|_W &= \left(\begin{array}{ccc}1 &  a & 0 \\ 0 &  1 & 0 \\ 0 & 0 & 1 \end{array}\right),\quad 
  D_\chi|_W = \left(\begin{array}{ccc}1 &  0& 0  \\ -a&  1& 0  \\ 0 & 0 & 1 \end{array}\right)\quad\text{and}\\
  D_\gamma|_W &= \left(\begin{array}{ccc}2\frac{bc}{a}+1  &  2\frac{c^2}{a}     & 0  \\                                                                                                                                                                                                                                                                                                                                                                       
                          -2\frac{b^2}{a}  & -2\frac{bc}{a}+1    & 0 \\ 
                           \frac{b}{a}     &  \frac{c}{a}        & 1
  \end{array}\right)
\end{align*} 
with respect to an appropriate basis, where \(a=22-4N-4M\), \(b=6+3m\) and \(c=12-3N+9m\). If we choose \(c=0\) or equivalently \(N=3m+4\), then the group generated by \(D_\delta|_W,~D_\chi|_W,~D_\gamma|_W\) contains a non-trivial element of the unipotent radical of the symplectic group on \(W\), namely \((D_\chi|_W)^{-2b^2}\circ (D_\gamma|_W)^{a^2}\) is represented by
  \[
    \begin{pmatrix}
    1 & 0 & 0 \\
    0 & 1 & 0 \\
    b\,a& 0 & 1
    \end{pmatrix}.
  \] 
\subsection{Zariski density and arithmeticity for a genus four family}  
At this point, we are ready to establish the arithmeticity of the Kontsevich--Zorich monodromy of \(\mathcal{O}_{N,M}^{(4)}\) for many choices of $N$, $M$. More precisely, consider the matrix 
  \[
  A_4(N,M):=M_h^{(0)}\cdot M_v^{(0)}\in\mathbb{R}^{6\times 6}.
  \] 
The characteristical polynomial of $A_4(N,M)$ is a reciprocal, sextic polynomial
  \[
  P(X)=\chi_A(X)=X^6+a_1\,X^5+a_2\,X^4+a_3\,X^3+a_2\,X^2+a_1\,X+1\in\mathbb{Z}[X]
  \]
and the coefficients $a_1,a_2,a_3\in\mathbb{R}$ are given by
  \begin{align*}
  a_1 & = 312\,m^2 + 650\,m + 238 \\
  a_2 & = 22032\,m^4 +98280\,m^3 +146568\,m^2 +84520\,m+15743\\
  a_3 & = 279936\,m^6 +2099520\,m^5 +6161184\,m^4 +8927280\,m^3\\ 
      & +6611328\,m^2 +2317980\,m+299812
  \end{align*}
We have  $1/X^3\cdot P(X)=Q(X+1/X+2)$ for the cubic polynomial 
\begin{align}\label{Eq:CubicQ}
  Q(Y)=Y^3+ (a_1-6)\,Y^2 +(-4\,a_1+a_2+9)\,Y+2a_1-2a_2+a_3-2
\end{align}
For $m\equiv 1$ modulo $13$, the polynomials $P(X)$ and $Q(Y)$ can be written by irreducible factors modulo $13$ as 
\begin{align*}
   P(X)& \equiv x^6 + 4\,X^5 + 10\,X^4 + 6\,X^3 + 10\,X^2 + 4\,X + 1 ~\text{modulo}~13 \quad \text{and}\\
   Q(Y)& \equiv X^3 +11\,X^2 +3\,X+5 ~\text{modulo}~13.
\end{align*}

In the sequel, we will assume $m\equiv 1$ modulo $13$ and thus the polynomials $P(X)$ and $Q(Y)$ are irreducible over the rational numbers $\mathbb{Q}$. 

As in Subsection \ref{SubSec:RecHyp} we will identify the Galois group $\text{Gal}(P)$ of the reciprocal degree $6$ polynomial $P(X)\in\mathbb{Z}[X]$ with a subgroup of the hyperoctahedral group $G_3$ as well as with a subgroup of the permutation group $S_6$ (see Subsection \ref{SubSec:GalGasPG}). The discriminant $\text{Disc}(Q)$ of the polynomial $Q(Y)\in\mathbb{Z}[Y]$ has an irreducible factorization in terms of $m$ as
\begin{displaymath}
\text{Disc}(Q)=
c_9\,\left(\sum_{i=0}^8 c_i\, m^i\right)(m+2)^2 \, (3\,m+4)^2,
\end{displaymath}
with coefficients 
\begin{IEEEeqnarray*}{rClCrClCrcl}
    c_9 &=& 186624,   &\quad & c_8 &=& 1778112, &\quad& c_7&=&7832160, \\
    c_6 &=& 14307444, &\quad & c_5 &=& 13909500,&\quad& c_4&=&8133701,\\
    c_3 &=& 2980770,  &\quad & c_2 &=& 676093,  &\quad& c_1&=& 87020,\\
    c_0 &=&4900.      &      &     & &          &     &    & &
\end{IEEEeqnarray*}

By Siegel's theorem on integral points of algebraic curves (see, e.g., \cite{hindry00}), we have that the discriminant of $Q(Y)$ is a square of a rational number only for finitely many choices of $m$. This implies that $\text{Gal}(Q)=S_3$ for all but finitely many $m\in\mathbb{N}$ with $m\equiv 1$ modulo $13$. Furthermore $\text{Gal}(P)$ the Galois-group of $P(X)\in\mathbb{Z}[X]$ is a subgroup of the hyperoctahedral group $G_3=\mathbb{Z}_2^3\rtimes S_3$ such that $\text{Gal}(P)$ projects surjectively onto $S_3$.

The only non-trivial subgroups of  $G_3=\mathbb{Z}_2^3\rtimes S_3$ which project surjectively onto $S_3$ are the groups $H_{3,1},H_{3,2}$ and $H_{3,3}$ defined in Lemma \ref{Prop:SubgrPhiOnto}. 

Next we want to factorize the expressions $\Delta_{3,1}=\delta_{3,1}^2$ and $\Delta_{3,2}=\delta_{3,2}^2$ from Lemma \ref{Lem:MaxSubDiscr} in terms of $m$. We get
\begin{align*}
\Delta_{3,1}=\delta_{3,1}^2 
= c_7 ( \sum_{i=0}^6 c_i\,m^i) ( 2\,m+1)( 3\,m+1)(m+2)^2 (3\,m+4)^2
\end{align*}
with coefficients
\begin{IEEEeqnarray*}{rClCrCLCrClCrCl}
c_7 &=&165888, &\quad& c_6 &=& 8748,   &\quad& c_5&=&65610, &\quad& c_4&=&191160,\\
c_3&=&272835,  &     & c_2 &=& 197463, &     & c_1&=&67195, &     & c_0&=&8400.   
\end{IEEEeqnarray*}

and $\Delta_{3,2}=\delta_{3,2}^2=\Delta_{3,1}\cdot \text{Disc}(Q)$. 

By applying Siegel's theorem again, we see that the expressions $\delta_{3,1}$ and $\delta_{3,2}$ are not rational numbers for all but finitely many $m\in\mathbb{N}$ with $m\equiv 1$ modulo $13$. In particular, the Galois group $\text{Gal}(P)$ of $P(X)$ is not contained in the subgroups $H_{3,1}$ or $H_{3,2}$ of the hyperoctahedral group $G_3$ by the fundamental theorem of Galois theory and Lemma \ref{Lem:MaxSubDiscr}.

Furthermore, if we have $m\equiv 1$ modulo $11$, the discriminant $\text{Disc}(P)$ of $P(X)\in\mathbb{Z}[X]$ is not divisible by $11$  since $\text{Disc}(P)=9~\text{modulo}~11$ and the polynomial $P(X)\in\mathbb{Z}[X]$ can be written by irreducible factors as
  \[
  P(X)\equiv (X^2 + 10\,X + 1)(X^4 + 2\,X^3 + 8\,X^2 + 2\,X + 1)~\text{modulo}~11.
  \]
If we view $\text{Gal}(P)$ as a subgroup of $\text{Sym}(\{\lambda_i,\lambda_i^{-1}\mid i=1,2,3\})$, then Dedekind's theorem (cf. \S\ref{SubSec:GalGasPG}) says that $\text{Gal}(P)$ contains a permutation of type $(2,4)$ for $m\equiv 1$ modulo $11$.
The groups $S_3$ and $H_{3,3}\leq G_3$ on the other hand contain only non-trivial permutations of cycle type $(6), (3,3), (2,2,2)$ or $(1,1,2,2)$ (see Appendix \ref{Sect:PermTypes}). Hence $\text{Gal}(P)$ is not contained in one of the groups $S_3$, or $H_{3,3}$ of $G_3$ for $m\equiv 1$ modulo $11$.

In summary, we showed the main part of the following proposition:

\begin{proposition}\label{Prop:GalPinG4}
For all but finitely many choices of $m\in\mathbb{N}$ such that $m\equiv 1$ modulo $p$, where $p\in\{11,\,13\}$, we have that $A_4(N,M)=M_h^{(0)}\cdot M_v^{(0)}\in\mathbb{R}^{6\times 6}$ is a Galois pinching matrix.
\end{proposition}

\begin{proof}
    The discriminant $\text{Disc}(Q(Y))$ of the cubic polynomial $Q(Y)\in\mathbb{Z}[Y]$ from (\ref{Eq:CubicQ}) converges to infinity for growing $m$. This shows that $Q(Y)$ has three distinct real roots for $m\equiv 1$ modulo $13$ big enough. Furthermore for $m$ big enough all the coefficients of the polynomial $Q(Y)$ are positive and hence by Decarte's rule of signs the three roots $\mu_1,~\mu_2$ and $\mu_3$ of $Q(Y)$ are negative. Furthermore
      \[
      \mu_i=\lambda_i+\lambda_i^{-1}+2\quad \text{for}~i=1,2,3,
      \]
    for the six roots $\{\lambda_i,~\lambda_i^{-1}\mid i=1,2,3\}$ of $P(X)$. With $\lambda_i^{-1}=\overline{\lambda_i}/|\lambda_i|$, we conclude for the imaginary part $\text{Im}(\mu_i)$ of $\mu_i$ for every $i=1,2,3$:
    \begin{align*}
        0=\text{Im}(\mu_i)=\text{Im}(\lambda_i)(1-1/|\lambda_i|)
    \end{align*}
    This shows $|\lambda_i|=1$ or $\text{Im}(\lambda_i)=0$ for every $i=1,2,3$. Assume that $\text{Im}(\lambda_j)\neq 0$ for some $j\in\{1,2,3\}$. Then $|\lambda_j|=1$ and 
      \[
      0>\text{Re}(\mu_j)=\text{Re}(\lambda_j+\lambda_j^{-1}+2)=\text{Re}(\lambda_j)+\text{Re}(\overline{\lambda_j})+2.
      \]
    This would imply $\text{Re}(\lambda_j)<-1$ a contradiction to $|\lambda_j|=1$.

    This shows that all roots of $P(X)$ are real for $m\equiv 1$ modulo $13$ big enough. Together with the calculations on $\text{Gal}(P)$ from this section we conclude that $A_4(N,M)\in\mathbb{R}^{6\times 6}$ is a Galois pinching matrix for every natural number $m$ big enough such that $m\equiv 1$ modulo $p$, where $p\in\{11,\,13\}$.
\end{proof}

From this statement, it is not hard to show that: 

\begin{theorem}\label{t.A.g4}
The Kontsevich-Zorich monodromies of the genus four origamis $\mathcal{O}_{N,M}^{(4)}\in\mathcal{H}(6)$ with $M=2m+4$ and $N=3m+4$ are finite index subgroups of the symplectic group $\text{Sp}(H_1^{(0)}(\mathcal{O}_{N,M}^{(4)},\mathbb{Z}))$ for all but perhaps finitely many $m\in\mathbb{N}$ such that $m\equiv 1$ modulo $p$, where $p\in\{11,\,13\}$.
\end{theorem}

\begin{proof}
One can check that the matrix $B\neq\text{Id}$ associated to an appropriate Dehn twist in the direction $(1,1)$ is a unipotent matrix such that the image
$(B-\text{Id})(\mathbb{R}^6)$ is one-dimensional and hence not a Lagrangian subspace (cf. the relevant matrix $B$ is called $M_{\delta}^{(0)}$ in Section 5.1 of \cite{bonnafoux22}). Since the matrix $A_4(N,M)$ is Galois pinching for all but finitely many choices of $m\in\mathbb{N}$ with $m\equiv 1$ modulo $p$, where $p\in\{11,\,13\}$ as in Proposition \ref{Prop:GalPinG4}. The desired result follows now from Singh--Venkataramana's arithmeticity criterion (cf. \S\ref{ss.general-strategy}).
\end{proof}

\section{Genus five stairs}\label{s.g5}

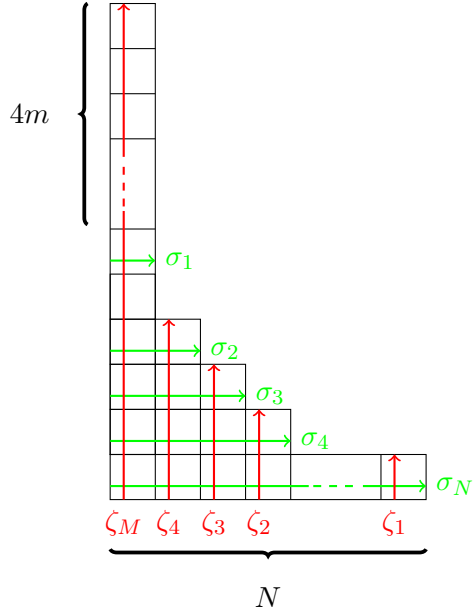
\begin{figure}
\begin{center}
\begin{tikzpicture}[scale=0.6]
	\draw (0,-1) rectangle (7,0);
	\draw (0,0) rectangle (4,1);
	\draw (0,1) rectangle (3,2);
	\draw (0,2) rectangle (2,3);
	\draw (0,3) rectangle (1,4);
	\draw (0,0) rectangle (1,10);
	\draw (1,-1) -- (1,3);
	\draw (2,-1) -- (2,2);
	\draw (3,-1) -- (3,1);
	\draw (4,-1) -- (4,0);
	\draw (6,-1) -- (6,0);
	\draw (0,5) -- (1,5);
	\draw (0,7) -- (1,7);
	\draw (0,8) -- (1,8);
	\draw (0,9) -- (1,9);
		
    
    \draw [decorate,line width=0.5mm,decoration={brace}] (-0.5,5.1) --  (-0.5,10) node[pos=0.5,left=10pt,black]{$4m$};
    \draw [decorate,line width=0.5mm,decoration={brace,mirror}] (0,-2.1) --  (7,-2.1) node[pos=0.5,below=10pt,black]{$N$};
    
     \draw[thick,color=green] (0,-0.7)--(4.4,-0.7);
     \draw[dashed,thick,color=green] (4.5,-0.7) -- (5.5,-0.7);
	\draw[thick,color=green, ->] (5.6,-0.7)--(7,-0.7);
	\node[color=green, right] (sigma7) at (7,-0.7) {\(\sigma_N\)};
	\draw[thick,color=green, ->] (0,0.3)--(4,0.3);
	\node[color=green, right] (sigma2) at (4,0.3) {\(\sigma_4\)};
	\draw[thick,color=green, ->] (0,1.3)--(3,1.3);
	\node[color=green, right] (sigma3) at (3,1.3) {\(\sigma_3\)};
	\draw[thick,color=green, ->] (0,2.3)--(2,2.3);
	\node[color=green, right] (sigma2) at (2,2.3) {\(\sigma_2\)};
	\draw[thick,color=green, ->] (0,4.3)--(1,4.3);
	\node[color=green, right] (sigma1) at (1,4.3) {\(\sigma_1\)};
	
	\draw[thick,color=red] (0.3,-1)--(0.3,5.4);
	\draw[dashed,thick,color=red] (0.3,5.5) -- (0.3,6.5);
	\draw[thick, color=red,->] (0.3,6.6) -- (0.3,10);
	\node[color=red, below] (zeta1) at (0.3,-1) {\(\zeta_M\)};
	\draw[thick,color=red, ->] (1.3,-1)--(1.3,3);
	\node[color=red, below] (zeta2) at (1.3,-1) {\(\zeta_4\)};
	\draw[thick,color=red, ->] (2.3,-1)--(2.3,2);
	\node[color=red, below] (zeta3) at (2.3,-1) {\(\zeta_3\)};
	\draw[thick,color=red, ->] (3.3,-1)--(3.3,1);
	\node[color=red, below] (zeta3) at (3.3,-1) {\(\zeta_2\)};
	\draw[thick,color=red, ->] (6.3,-1)--(6.3,0);
    \node[color=red,below] (zeta4) at (6.3,-1) {\(\zeta_1\)};
	
\end{tikzpicture}
\end{center}

\caption{Origami \(\mathcal{O}_{N,M}^{(5)}\) with horizontal waist curves \(\sigma_1,\sigma_2,\sigma_3,\sigma_4,\sigma_N\) and vertical waist curves \(\zeta_1,\zeta_2,\zeta_3,\zeta_4,\zeta_M\).}
\label{fig:OrigamiNM5}
\end{figure}

For $N,M\in\mathbb{N}$ with $M=6+4m~(m\in\mathbb{N})$, we consider the origami $\mathcal{O}^{(5)}_{N,M}$ that is given by the following horizontal and vertical permutation $h,v\in\text{Sym}(\{1,2,\dots, N+M+5\})$:
\begin{align*}
  h = &  (1,\dots, N)(N+1,\dots,N+4)(N+5,\dots,N+7)\\
      &  (N+8,N+9)(N+10)\dots(N+M+5)\\[2mm]
  v = &  (1,N+1,N+5,N+8,N+10,\dots,N+M+5)\\
      &  (2,N+2,N+6,N+9)(3,N+3,N+7)(4,N+4)(5)\dots(N)
\end{align*}
The five waist curves $\sigma_1,\dots,\sigma_4,\sigma_N$ of the maximal horizontal cylinders together with the waist curves $\zeta_1,\dots,\zeta_5,\zeta_M$ of the maximal vertical cylinders form a basis $B$ of the absolute homology $H_1(\mathcal{O}^{(5)}_{N,M},\mathbb{Q})$ of the origami $\mathcal{O}^{(5)}_{N,M}$ see Figure \ref{fig:OrigamiNM5}. With respect to the basis $B$ the symplectic intersection form $\Omega$ on $H_1(\mathcal{O}^{(5)}_{N,M},\mathbb{Q})$ has a matrix representation $M\Omega=(\Omega(\sigma_i,\zeta_j)_{i,j}$given by
\begin{displaymath}
 M\Omega=
  \left(\begin{array}{cccccccccc}
      0 & 0 & 0 & 0 & 0 & 0 & 0 & 0 & 0 & 1 \\
      0 & 0 & 0 & 0 & 0 & 0 & 0 & 0 & 1 & 1 \\
      0 & 0 & 0 & 0 & 0 & 0 & 0 & 1 & 1 & 1 \\
      0 & 0 & 0 & 0 & 0 & 0 & 1 & 1 & 1 & 1 \\
      0 & 0 & 0 & 0 & 0 & 1 & 1 & 1 & 1 & 1 \\
      0 & 0 & 0 & 0 &-1 & 0 & 0 & 0 & 0 & 0 \\
      0 & 0 & 0 &-1 &-1 & 0 & 0 & 0 & 0 & 0 \\
      0 & 0 &-1 &-1 &-1 & 0 & 0 & 0 & 0 & 0 \\
      0 &-1 &-1 &-1 &-1 & 0 & 0 & 0 & 0 & 0 \\
     -1 &-1 &-1 &-1 &-1 & 0 & 0 & 0 & 0 & 0
  \end{array}\right).
\end{displaymath}
If we compare the length of the waist curves of the five maximal horizontal and vertical cylinders of $\mathcal{O}^{(5)}_{N,M}$, we see that 
  \[
  B^{(0)}=\{\Sigma_1,\,\Sigma_2,\,\Sigma_3,\,\Sigma_N,\,Z_1,\,Z_2,\,Z_3,\,Z_M\}
  \]
is a basis of the non-tautological part $H_1^{(0)}(\mathcal{O}_{N,M}^{(5)},\mathbb{Q})$, where
\begin{IEEEeqnarray*}{rClCrClCrClCrCl}
\Sigma_1&:=&\sigma_2 - 2\,\sigma_1, &\quad& \Sigma_2 &:=&\sigma_3-3\,\sigma_1,&\quad &
\Sigma_3&:=&\sigma_4-4\,\sigma_1,   &\quad& \Sigma_N &:=&\sigma_N-N\,\sigma_1,\\
     Z_1&:=&\zeta_2 - 2\,\zeta_1,   &     &       Z_2&:=&\zeta_3-3\,\zeta_1,  &   &
     Z_3&:=&\zeta_4-4\,\zeta_1,     &     &       Z_M&:=&\zeta_M-M\,\zeta_1.
\end{IEEEeqnarray*}

If we restrict the intersection form $\Omega$ to subspace $H_1^{(0)}(\mathcal{O}_{N,M}^{(5)},\mathbb{Q})$ of the absolute homology then it can be represented by the following matrix $M\Omega^{(0)}=(\Omega|{H_1^{(0)}}(\Sigma_i,Z_j))_{i,j}$ with respect to the basis $B^{(0)}$ from above:
\begin{align*}
  M\Omega^{(0)}=
      \left(\begin{array}{cccccccc}
      0 & 0 & 0 & 0 & 0 & 0 & 1 & -1\\
      0 & 0 & 0 & 0 & 0 & 1 & 1 & -2\\
      0 & 0 & 0 & 0 & 1 & 1 & 1 & -3\\
      0 & 0 & 0 & 0 & -1 & -2 & -3 & 1-N-M\\
      0 & 0 & -1 & 1 & 0 & 0 & 0 & 0\\
      0 & -1 & -1 & 2 & 0 & 0 & 0 & 0\\
     -1 & -1 & -1 & 3 & 0 & 0 & 0 & 0\\
      1 & 2 & 3 & M+N-1 & 0 & 0 & 0 & 0
      \end{array}\right)
\end{align*}

\subsection{Dehn twists in genus five}\label{Sec:DehnTwistg=5}

Before we start with our calcaulations for the origami $\mathcal{O}_{N,M}^{(5)}$, we want to explain what we mean by length or combinatorial lenght of a curve in an origami.

\begin{remark}
    For an origami $\pi\colon\mathcal{O}\to \mathbb{R}^2/\mathbb{Z}^2$ and a curve $\gamma\colon[0,1]\to \mathcal{O}$ we mean by \text{(combinatorial) length} the number of $t\in(0,1]$ such that $\pi(\gamma(t))=\pi(\gamma(0))$.
\end{remark}

Now we can start with our calculations.
Recall that $M=6+4m$. In this case, the cylinder decompositions of $\mathcal{O}_{N,M}^{(5)}$ in the directions $(1,2)$, $(1,-2)$ and $(1,4)$ have the following structure. 


\begin{figure}
\centering
\noindent
\begin{subfigure}[b]{0.5\textwidth}
\centering
\begin{tikzpicture}[scale=0.6]
	\draw (0,-1) rectangle (4,0);
	\draw (6,-1) rectangle (7,0);
	\draw (0,-1) rectangle (3,2);
	\draw (0,-1) rectangle (2,3);
	\draw (0,-1) rectangle (1,5);
	\draw (0,-1) rectangle (4,1);
	\draw (0,7) rectangle (1,11);
	\draw (1,0) -- (1,3);
	\draw (2,0) -- (2,2);
	\draw (3,0) -- (3,1);
	\draw (0,4) -- (1,4);
	\draw (0,7) -- (1,7);
	\draw (0,8) -- (1,8);
	\draw (0,9) -- (1,9);
	\draw (0,10) -- (1,10);
    \draw [decorate,line width=0.5mm,decoration={brace}] (-0.5,-1) --  (-0.5,11) node[pos=0.5,left=10pt,black, rotate=-90]{$M=6+4m$};
    \draw [decorate,line width=0.5mm,decoration={brace,mirror}] (0,-2.1) --  (7,-2.1) node[pos=0.5,below=10pt,black]{$N$};
    \draw [dashed] (4.8,-0.5) -- (5.2,-0.5);
    \draw [dashed] (0.5,5.8) -- (0.5,6.2);
    \draw [decorate, decoration={snake}] (0,5.5) -- (1,5.5);
    \draw [decorate, decoration={snake}] (0,6.5) -- (1,6.5);
    \draw (0,5) -- (0,5.5);
    \draw (1,5) -- (1,5.5);
    \draw (0,6.5) -- (0,7);
    \draw (1,6.5) -- (1,7);
    \draw [decorate, decoration={snake}] (4.5,-1) -- (4.5,0);
    \draw [decorate, decoration={snake}] (5.5,-1) -- (5.5,0);
    \draw (4,-1) -- (4.5,-1);
    \draw (4,0) -- (4.5,0);
    \draw (5.5,-1) -- (6,-1);
    \draw (5.5,0) -- (6,0);
    \draw[pattern color=blue, pattern = north east lines] (0,-1) -- (0.5,-1) -- (2,2) -- (2,3) -- (0,-1) --
     cycle;
    \draw[pattern color=blue, pattern = north east lines] (0,2) -- (1,4) -- (1,5) -- (0,3) -- (0,2) --
    cycle;
    \draw[pattern color=blue, pattern = north east lines] (0,4) -- (0.5,5) -- (0,5) -- (0,4) -- 
    cycle;
     \draw[pattern color=blue, pattern = north east lines] (0,7) -- (0.5,7) -- (1,8) -- (1,9) -- (0,7) --
    cycle;
     \draw[pattern color=blue, pattern = north east lines] (0,8) -- (1,10) -- (1,11) -- (0,9) -- (0,8) --
    cycle;
     \draw[pattern color=blue, pattern = north east lines] (0,10) -- (0.5,11) -- (0,11) -- (0,10) -- 
    cycle;
   \draw[thick, color=red] (0,10) -- (0.5,11);
   \draw[thick, color=red] (0,9) -- (1,11);
   \draw[thick, color=red] (0,8) -- (1,10);
   \draw[thick, color=red] (0,7) -- (1,9);
   \draw[thick, color=red] (0.5,7) -- (1,8);
   \draw[thick, color=red] (0,4) -- (0.5,5);
   \draw[thick, color=red] (0,3) -- (1,5);
   \draw[thick, color=red] (0,2) -- (1,4);
   \draw[thick, color=red] (0,1) -- (1,3);
   \draw[thick, color=red] (0,0) -- (1.5,3);
   \draw[thick, color=red] (0,-1) -- (2,3);
   \draw[thick, color=red] (0.5,-1) -- (2,2);
   \draw[thick, color=red] (1,-1) -- (2.5,2);
   \draw[thick, color=red] (1.5,-1) -- (3,2);
   \draw[thick, color=red] (2,-1) -- (3,1);
   \draw[thick, color=red] (2.5,-1) -- (3.5,1);
   \draw[thick, color=red] (3,-1) -- (4,1);
   \draw[thick, color=red] (3.5,-1) -- (4,0);
   \draw[thick, color=red] (6,-1) -- (6.5,0);
   \draw[thick, color=red] (6.5,-1) -- (7,0);
\end{tikzpicture}
\caption{}\label{fig:G5cylinder(1,2)}
\end{subfigure}%
\hfill
\begin{subfigure}[b]{0.5\textwidth}
\centering
\begin{tikzpicture}[scale=0.6]
	\draw (0,-1) rectangle (4,0);
	\draw (6,-1) rectangle (7,0);
	\draw (0,-1) rectangle (3,2);
	\draw (0,-1) rectangle (2,3);
	\draw (0,-1) rectangle (1,5);
	\draw (0,-1) rectangle (4,1);
	\draw (0,7) rectangle (1,11);
	\draw (1,0) -- (1,3);
	\draw (2,0) -- (2,2);
	\draw (3,0) -- (3,1);
	\draw (0,4) -- (1,4);
	\draw (0,7) -- (1,7);
	\draw (0,8) -- (1,8);
	\draw (0,9) -- (1,9);
	\draw (0,10) -- (1,10);
    \draw [decorate,line width=0.5mm,decoration={brace}] (-0.5,-1) --  (-0.5,11) node[pos=0.5,left=10pt,black,rotate=-90]{$M=6+4m$};
    \draw [decorate,line width=0.5mm,decoration={brace,mirror}] (0,-2.1) --  (7,-2.1) node[pos=0.5,below=10pt,black]{$N$};
    \draw [dashed] (4.8,-0.5) -- (5.2,-0.5);
    \draw [dashed] (0.5,5.8) -- (0.5,6.2);
    \draw [decorate, decoration={snake}] (0,5.5) -- (1,5.5);
    \draw [decorate, decoration={snake}] (0,6.5) -- (1,6.5);
    \draw (0,5) -- (0,5.5);
    \draw (1,5) -- (1,5.5);
    \draw (0,6.5) -- (0,7);
    \draw (1,6.5) -- (1,7);
    \draw [decorate, decoration={snake}] (4.5,-1) -- (4.5,0);
    \draw [decorate, decoration={snake}] (5.5,-1) -- (5.5,0);
    \draw (4,-1) -- (4.5,-1);
    \draw (4,0) -- (4.5,0);
    \draw (5.5,-1) -- (6,-1);
    \draw (5.5,0) -- (6,0);
    \draw[pattern color=blue, pattern = north east lines] (0,3) -- (0,2) -- (1.5,-1) -- (2,-1) -- (0,3) --
     cycle;
     \draw[pattern color=blue, pattern = north east lines] (1.5,3) -- (2,2) -- (2,3) -- (1.5,3) --
     cycle;
   \draw[thick, color=red] (0.5,11) -- (1,10);
   \draw[thick, color=red] (0,11) -- (1,9);
   \draw[thick, color=red] (0,10) -- (1,8);
   \draw[thick, color=red] (0,9) -- (1,7);
   \draw[thick, color=red] (0,8) -- (0.5,7);
   \draw[thick, color=red] (0.5,5) -- (1,4);
   \draw[thick, color=red] (0,5) -- (3,-1);
   \draw[thick, color=red] (0,4) -- (2.5,-1);
   \draw[thick, color=red] (0,3) -- (2,-1);
   \draw[thick, color=red] (0,2) -- (1.5,-1);
   \draw[thick, color=red] (0,1) -- (1,-1);
   \draw[thick, color=red] (0,0) -- (0.5,-1);
   \draw[thick, color=red] (1.5,3) -- (3.5,-1);
   \draw[thick, color=red] (2.5,2) -- (4,-1);
   \draw[thick, color=red] (3.5,1) -- (4,0);
   \draw[thick, color=red] (6,0) -- (6.5,-1);
   \draw[thick, color=red] (6.5,0) -- (7,-1);
\end{tikzpicture}
		\caption{}\label{fig:G5cylinder(1,-2)}
	\end{subfigure}
	\caption{Cylinder decomposition in direction $(1,2)$ and direction $(1,-2)$ of the origami $\mathcal{O}_{N,M}^{(5)}$. Here $\gamma_1$ is the waist curve of the blue cylinder in direction $(1,2)$ and $\alpha_1$ is the waist curve of the blue cylinder in direction $(1,-2)$.}
\end{figure}
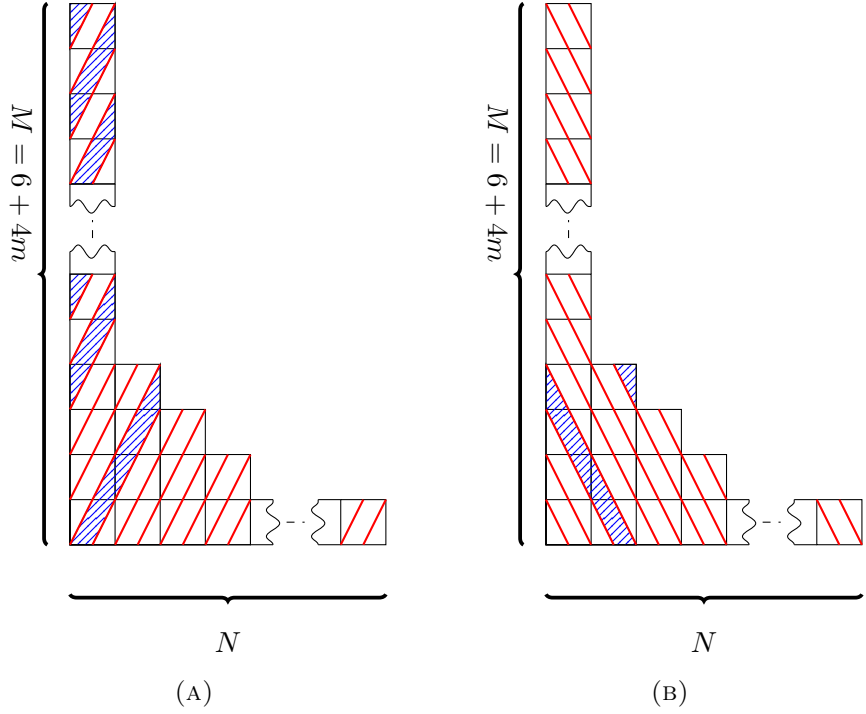


In the direction $(1,2)$, we find a waist curve $\gamma_1$ of length $3+2m$ and a waist curve $\gamma_2$ of length $8+N+2m$. Thus, $\Gamma:=(8+N+2m)\gamma_1-(3+2m)\gamma_2\in H_1^{(0)}(\mathcal{O}_{N,M}^{(5)},\mathbb{Z})$ and this direction yields a transvection
  \[
  D_\gamma\colon v\longmapsto v +(8+N+2m)\,\Omega(\gamma_1,v)\,\gamma_1 
                                  +(3+2m)\,\Omega(\gamma_2,v)\,\gamma_2.
  \]
  For later reference, let us observe that: 
\begin{IEEEeqnarray*}{lClClClClCl}
\Omega(\gamma_1,\sigma_1) &=&-1, &\quad& \Omega(\gamma_1,\sigma_2)&=&-1, &\quad&    
\Omega(\gamma_1,\sigma_3) &=&-1, \\
 \Omega(\gamma_1,\sigma_4)&=&-1, &\quad& \Omega(\gamma_1,\sigma_N)&=&-1, &\quad&  && \\          
\Omega(\gamma_1,\zeta_1) &=& 0,  &\quad& \Omega(\gamma_1,\zeta_2) &=& 0, &\quad&
\Omega(\gamma_1,\zeta_3) &=&0,   \\
 \Omega(\gamma_1,\zeta_4) &=&1, &\quad& \Omega(\gamma_1,\zeta_M) &=&2+2m &\quad&  && \\
\end{IEEEeqnarray*}
and
\begin{IEEEeqnarray*}{lClClClClCl}
\Omega(\gamma_2,\sigma_1)&=&-1,  &\quad& \Omega(\gamma_2,\sigma_2)&=&-3, &\quad& 
\Omega(\gamma_2,\sigma_3)&=&-5,  \\
\Omega(\gamma_2,\sigma_4)&=&-7,  &\quad& \Omega(\gamma_2,\sigma_N)&=&1-2N, &\quad&  & &\\          
\Omega(\gamma_2,\zeta_1) &=&1,   &\quad& \Omega(\gamma_2,\zeta_2) &=&2,    &\quad&  
\Omega(\gamma_2,\zeta_3) &=&3,   \\
\Omega(\gamma_2,\zeta_4) &=&3,   &\quad& \Omega(\gamma_2,\zeta_M) &=&4+2m  &\quad&  & &\\
\end{IEEEeqnarray*}

We can write $\Gamma\in H_1^{(0)}(\mathcal{O}_{N,M},\mathbb{Z})$ as a linear combination of elements of $B^{(0)}$ in the following way:
\begin{align*}
 \Gamma=&-(N+2m+8)\,\Sigma_1+(2m+3)\,\Sigma_2+(2m+3)\,\Sigma_3+(2m+3)\,\Sigma_N \\
        &+M\;Z_1+M\,Z_2+M\,Z_3-(N+5)\,Z_M
\end{align*}
Furthermore we calculate
\begin{equation}
    \begin{split}
    \begin{IEEEeqnarraybox}[][c]{rClCrCl}\label{Eq:DTgamma}
    D_\gamma(\Sigma_1) &=&\Sigma_1+ \Gamma, &\quad& D_\gamma(\Sigma_2)&=& \Sigma_2+ 2\Gamma, \\
    D_\gamma(\Sigma_3) &=&\Sigma_3+3\Gamma, &\quad& D_\gamma(\Sigma_N)&=& \Sigma_N+(N-1)\Gamma, \\
    D_\gamma(Z_1)      &=& Z_1,             &\quad& D_\gamma(Z_2)     &=& Z_2,  \\
    D_\gamma(Z_3)      &=& Z_3+\Gamma,      &\quad& D_\gamma(Z_M)     &=& Z_M+ (2+2m)\Gamma.
    \end{IEEEeqnarraybox}
    \end{split}
\end{equation}

In the direction $(1,-2)$, we have a waist curve $\alpha_1$ of length $2$ and a waist curve $\alpha_2$ of length $9+N+4m$. This yields a transvection
  \[
  D_\alpha\colon v \longmapsto v + (9+N+4m)\,\Omega(\alpha_1,v)\,\alpha_1
                                         +2\,\Omega(\alpha_2,v)\,\alpha_2.
  \] 
Again, for later reference, we note that: 
\begin{IEEEeqnarray*}{lClClClClCl}
\Omega(\alpha_1,\sigma_1)&=& 0, &\quad& \Omega(\alpha_1,\sigma_2)&=& 1, &\quad&
\Omega(\alpha_1,\sigma_3)&=& 1, \\
\Omega(\alpha_1,\sigma_4)&=& 1, &\quad& \Omega(\alpha_1,\sigma_N)&=&1, &\quad&  & &  \\
\Omega(\alpha_1,\zeta_1) &=& 0, &\quad& \Omega(\alpha_1,\zeta_2) &=&0, &\quad&
\Omega(\alpha_1,\zeta_3) &=& 0,  \\  
\Omega(\alpha_1,\zeta_4) &=& 1, &\quad& \Omega(\alpha_1,\zeta_M)&=&1,  &\quad&  & & 
\end{IEEEeqnarray*}
and
\begin{IEEEeqnarray*}{lClClClClCl}
\Omega(\alpha_2,\sigma_1)&=&2,  &\quad& \Omega(\alpha_2,\sigma_2)&=&3, &\quad&  
\Omega(\alpha_2,\sigma_3)&=&5,  \\
\Omega(\alpha_2,\sigma_4)&=&7,  &\quad& \Omega(\alpha_2,\sigma_N)&=&2N-1, &\quad& & & \\            \Omega(\alpha_2,\zeta_1) &=&1,  &\quad& \Omega(\alpha_2,\zeta_2) &=&2,    &\quad&
\Omega(\alpha_2,\zeta_3) &=&3,  \\
\Omega(\alpha_2,\zeta_4) &=&3,  &\quad& \Omega(\alpha_2,\zeta_M) &=&5+4m, &\quad& & &
\end{IEEEeqnarray*}
We can write $A=(9+N+4M)\alpha_1-2\alpha_2\in H_1^{(0)}(\mathcal{O}_{N,M}^{(5)},\mathbb{Z})$ as a linear combination of elements of $B^{(0)}$ as 
  \begin{displaymath}
  A=-(9+N+4m)\Sigma_1+2\,\Sigma_2+2\,\Sigma_3+2\Sigma_N-4\,Z_1-4\,Z_2+(7+N+4m)\,Z_3-4\,Z_M.
  \end{displaymath}
We calculate
\begin{equation}
    \begin{split}
    \begin{IEEEeqnarraybox}[][c]{rClCrCl}\label{Eq:DTalpha}
    D_\alpha(\Sigma_1) & = & \Sigma_1 + A, &\quad&  D_\alpha(\Sigma_2) & = & \Sigma_2 + A, \\
    D_\alpha(\Sigma_3) & = & \Sigma_3 + A, &     &  D_\alpha(\Sigma_N) & = & \Sigma_N + A, \\
    D_\alpha(Z_1)      & = & Z_1 ,         &     &  D_\alpha(Z_2)      & = & Z_2, \\
    D_\alpha(Z_3)      & = & Z_3 + A,      &     &  D_\alpha(Z_M)      & = & Z_M + A.
   \end{IEEEeqnarraybox}
   \end{split}
\end{equation}

This leads to a matrix representation $M_\alpha^{(0)}$ of $D_\alpha$on $H_1^{(0)}(\mathcal{O}_{N,M}^{(5)},\mathbb{Q})$ with respect to the basis $B^{(0)}$. It can be seen in Appendix \ref{Sec:DehnTwist}.


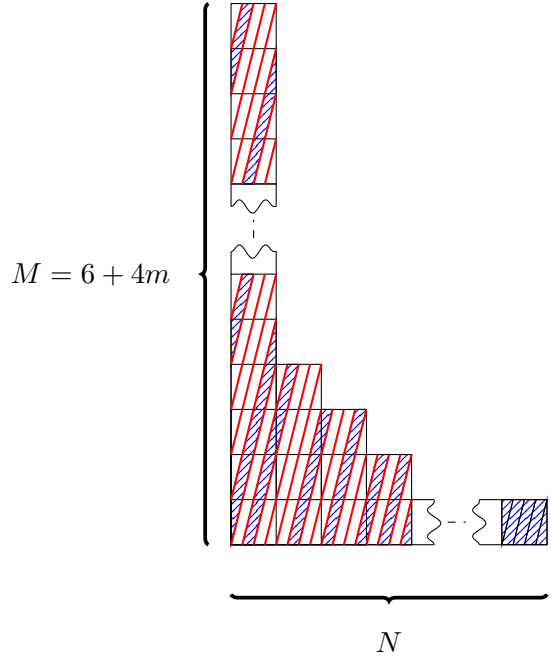
\begin{figure}[htb]
\centering
\begin{tikzpicture}[scale=0.6]
	\draw (0,-1) rectangle (4,0);
	\draw (6,-1) rectangle (7,0);
	\draw (0,-1) rectangle (3,2);
	\draw (0,-1) rectangle (2,3);
	\draw (0,-1) rectangle (1,5);
	\draw (0,-1) rectangle (4,1);
	\draw (0,7) rectangle (1,11);
	\draw (1,0) -- (1,3);
	\draw (2,0) -- (2,2);
	\draw (3,0) -- (3,1);
	\draw (0,4) -- (1,4);
	\draw (0,7) -- (1,7);
	\draw (0,8) -- (1,8);
	\draw (0,9) -- (1,9);
	\draw (0,10) -- (1,10);
    \draw [decorate,line width=0.5mm,decoration={brace}] (-0.5,-1) --  (-0.5,11) node[pos=0.5,left=10pt,black]{$M=6+4m$};
    \draw [decorate,line width=0.5mm,decoration={brace,mirror}] (0,-2.1) --  (7,-2.1) node[pos=0.5,below=10pt,black]{$N$};
    \draw [dashed] (4.8,-0.5) -- (5.2,-0.5);
    \draw [dashed] (0.5,5.8) -- (0.5,6.2);
    \draw [decorate, decoration={snake}] (0,5.5) -- (1,5.5);
    \draw [decorate, decoration={snake}] (0,6.5) -- (1,6.5);
    \draw (0,5) -- (0,5.5);
    \draw (1,5) -- (1,5.5);
    \draw (0,6.5) -- (0,7);
    \draw (1,6.5) -- (1,7);
    \draw [decorate, decoration={snake}] (4.5,-1) -- (4.5,0);
    \draw [decorate, decoration={snake}] (5.5,-1) -- (5.5,0);
    \draw (4,-1) -- (4.5,-1);
    \draw (4,0) -- (4.5,0);
    \draw (5.5,-1) -- (6,-1);
    \draw (5.5,0) -- (6,0);
    \draw[pattern color=blue, pattern = north east lines] (0,10) -- (0,9) -- (0.5,11) -- (0.25,11) -- (0,10)    --cycle;
    \draw[pattern color=blue, pattern = north east lines] (0.25,7) -- (0.5,7) -- (1,9) -- (1,10) -- (0.25,7)    --cycle;
    \draw[pattern color=blue, pattern = north east lines] (0,3) -- (0.5,5) -- (0.25,5) -- (0,4) -- (0,3)    --cycle;
    \draw[pattern color=blue, pattern = north east lines] (0,-1) -- (1,3) -- (1,4) -- (0,0) -- (0,-1)    --cycle;
    \draw[pattern color=blue, pattern = north east lines] (0.25,-1) -- (0.5,-1) -- (1.5,3) -- (1.25,3) -- (0.25,-1) --cycle;
    \draw[pattern color=blue, pattern = north east lines] (1.25,-1) -- (1.5,-1) -- (2.25,2) -- (2,2) -- (1.25,-1) --cycle;
    \draw[pattern color=blue, pattern = north east lines] (2,-1) -- (2.25,-1) -- (3,2) -- (2.75,2) -- (2,-1) --cycle;
    \draw[pattern color=blue, pattern = north east lines] (2.75,-1) -- (3,-1) -- (3.5,1) -- (3.25,1) -- (2.75,-1) --cycle;
    \draw[pattern color=blue, pattern = north east lines] (3.25,-1) -- (3.5,-1) -- (4,1) -- (3.75,1) -- (3.25,-1) --cycle;
    \draw[pattern color=blue, pattern = north east lines] (3.75,-1) -- (4,-1) -- (4,0) -- (3.75,-1) --cycle;
    \draw[pattern color=blue, pattern = north east lines] (6,-1) -- (6.25,-0) -- (6,0) -- (6,-1) --cycle;
    \draw[pattern color=blue, pattern = north east lines] (6,-1) -- (6.25,-1) -- (6.5,0) -- (6.25,0) -- (6,-1) --cycle;
    \draw[pattern color=blue, pattern = north east lines] (6.25,-1) -- (6.5,-1) -- (6.75,0) -- (6.5,0) -- (6.25,-1) --cycle;
    \draw[pattern color=blue, pattern = north east lines] (6.5,-1) -- (6.75,-1) -- (7,0) -- (6.75,0) -- (6.5,-1) --cycle;
    \draw[pattern color=blue, pattern = north east lines] (6.75,-1) -- (7,-1) -- (7,0) -- (6.75,-1) --cycle;
   \draw[thick, color=red] (0,10) -- (0.25,11);
   \draw[thick, color=red] (0,9) -- (0.5,11);
   \draw[thick, color=red] (0,8) -- (0.75,11);
   \draw[thick, color=red] (0,7) -- (1,11);
   \draw[thick, color=red] (0.25,7) -- (1,10);
   \draw[thick, color=red] (0.5,7) -- (1,9);
   \draw[thick, color=red] (0.75,7) -- (1,8);
   \draw[thick, color=red] (0,4) -- (0.25,5);
   \draw[thick, color=red] (0,3) -- (0.5,5);
   \draw[thick, color=red] (0,2) -- (0.75,5);
   \draw[thick, color=red] (0,1) -- (1,5);
   \draw[thick, color=red] (0,0) -- (1,4);
   \draw[thick, color=red] (0,-1) -- (1,3);
   \draw[thick, color=red] (0.25,-1) -- (1.25,3);
   \draw[thick, color=red] (0.5,-1) -- (1.5,3);
   \draw[thick, color=red] (0.75,-1) -- (1.75,3);
   \draw[thick, color=red] (1,-1) -- (2,3);
   \draw[thick, color=red] (1.25,-1) -- (2,2);
   \draw[thick, color=red] (1.5,-1) -- (2.25,2);
   \draw[thick, color=red] (1.75,-1) -- (2.5,2);
   \draw[thick, color=red] (2,-1) -- (2.75,2);
   \draw[thick, color=red] (2.25,-1) -- (3,2);
   \draw[thick, color=red] (2.5,-1) -- (3,1);
   \draw[thick, color=red] (2.75,-1) -- (3.25,1);
   \draw[thick, color=red] (3,-1) -- (3.5,1);
   \draw[thick, color=red] (3.25,-1) -- (3.75,1);
   \draw[thick, color=red] (3.5,-1) -- (4,1);
   \draw[thick, color=red] (3.75,-1) -- (4,0);
\end{tikzpicture}
\caption{Origami \(\mathcal{O}^{(5)}_{N,M}\) with cylinder decomposition in direction \((1,4)\). Here $\chi_1$ is the waist curve of the blue cylinder.}
\label{fig:G5cylinder(1,4)}
\end{figure}


Finally, in the direction $(1,4)$, we find a waist curve $\chi_1$ of length $1+N+m$ and a waist curve $\chi_2$ of length $10+3m$. The transvection associated to this direction is: 
  \[
  D_\chi\colon v \longmapsto v + (10+3m)\,\Omega(\chi_1,v)\,\chi_1
                                +(1+N+m)\,\Omega(\chi_2,v)\,\chi_2
  \]
  Also, let us remark that: 
\begin{equation}
  \begin{split}
  \begin{IEEEeqnarraybox}[]{lClClClClCl}
    \Omega(\chi_1,\sigma_1)&=&-1, &\quad& \Omega(\chi_1,\sigma_2)&=&-2,        &\quad&
    \Omega(\chi_1,\sigma_3)&=&-4, \\
    \Omega(\chi_1,\sigma_4)&=&-6, &\quad& \Omega(\chi_1,\sigma_N)&=&-6-4(N-4), &\quad&  && \\
    \Omega(\chi_1,\zeta_1) &=& 1,  &\quad& \Omega(\chi_1,\zeta_2) &=&1,          &\quad&
    \Omega(\chi_1,\zeta_3) &=&1, \\
    \Omega(\chi_1,\zeta_4) &=&1,  &\quad&\Omega(\chi_1,\zeta_M)&=&2+m,           &\quad& &&
 \end{IEEEeqnarraybox}
\end{split}
\end{equation}
as well as
\begin{equation}
  \begin{split}
  \begin{IEEEeqnarraybox}[]{lClClClClCl}
    \Omega(\chi_2,\sigma_1)&=&-3, &\quad& \Omega(\chi_2,\sigma_2)&=&-6,    &\quad& 
    \Omega(\chi_2,\sigma_3)&=&-8,   \\
    \Omega(\chi_2,\sigma_4)&=&-10,  &\quad& \Omega(\chi_2,\sigma_N)&=&-10, &\quad& && \\            \Omega(\chi_2,\zeta_1) &=&0,    &\quad& \Omega(\chi_2,\zeta_2)&=&1,    &\quad&
    \Omega(\chi_2,\zeta_3) &=&2,    \\
    \Omega(\chi_2,\zeta_4) &=&3,    &\quad&  \Omega(\chi_2,\zeta_M)&=&4+3m.  &\quad& &&               \end{IEEEeqnarraybox}
  \end{split}
\end{equation}
We can write $X=(3m+10)\,\chi_1-(N+m+1)\,\chi_2\in H_1^{(0)}(\mathcal{O}_{N,M}^{(5)},\mathbb{Z})$ as a linear combination of elements of $B^{(0)}$ as
  \begin{align*}
  X=&\quad (N+m+1)\,\Sigma_1+(N+m+1)\,\Sigma_2+(N+m+1)\Sigma_3-(3m+10)\Sigma_N\\
    &+ (2N-4m-18)\,Z_1+(2N-4m-18)\,Z_2+(3N-7)\,Z_3+ (3N-7)\, Z_M.
  \end{align*}
We calculate for the image of $B^{(0)}$ under $D_\chi$:
\begin{equation}
  \begin{split}
    \begin{IEEEeqnarraybox}[][c]{lClClCl}\label{Eq:DTchi}
    D_\chi(\Sigma_1) & = & \Sigma_1,     &\quad& D_\chi(\Sigma_2) & = & \Sigma_2 - X, \\
    D_\chi(\Sigma_3) & = & \Sigma_3-2X,  &     & D_\chi(\Sigma_N) & = & \Sigma_N - (3N-10)X, \\
    D_\chi(Z_1)      & = & Z_1-X,         &     & D_\chi(Z_2)      & = & Z_2-2X,   \\
    D_\chi(Z_3)      & = & Z_3-3X,        &     & D_\chi(Z_M)      & = & Z_M-(4+3m)X.  
    \end{IEEEeqnarraybox}
  \end{split}
\end{equation}


\begin{figure}
\centering
\noindent
\begin{subfigure}[b]{0.5\textwidth}
\centering
\begin{tikzpicture}[scale=0.6]
	\draw (0,-1) rectangle (4,0);
	\draw (6,-1) rectangle (7,0);
	\draw (0,-1) rectangle (3,2);
	\draw (0,-1) rectangle (2,3);
	\draw (0,-1) rectangle (1,5);
	\draw (0,-1) rectangle (4,1);
	\draw (0,7) rectangle (1,11);
	\draw (1,0) -- (1,3);
	\draw (2,0) -- (2,2);
	\draw (3,0) -- (3,1);
	\draw (0,4) -- (1,4);
	\draw (0,7) -- (1,7);
	\draw (0,8) -- (1,8);
	\draw (0,9) -- (1,9);
	\draw (0,10) -- (1,10);
    \draw [decorate,line width=0.5mm,decoration={brace}] (-0.5,-1) --  (-0.5,11) node[pos=0.5,left=10pt,black,rotate=-90]{$M=6+4m$};
    \draw [decorate,line width=0.5mm,decoration={brace,mirror}] (0,-2.1) --  (7,-2.1) node[pos=0.5,below=10pt,black]{$N$};
    \draw [dashed] (4.8,-0.5) -- (5.2,-0.5);
    \draw [dashed] (0.5,5.8) -- (0.5,6.2);
    \draw [decorate, decoration={snake}] (0,5.5) -- (1,5.5);
    \draw [decorate, decoration={snake}] (0,6.5) -- (1,6.5);
    \draw (0,5) -- (0,5.5);
    \draw (1,5) -- (1,5.5);
    \draw (0,6.5) -- (0,7);
    \draw (1,6.5) -- (1,7);
    \draw [decorate, decoration={snake}] (4.5,-1) -- (4.5,0);
    \draw [decorate, decoration={snake}] (5.5,-1) -- (5.5,0);
    \draw (4,-1) -- (4.5,-1);
    \draw (4,0) -- (4.5,0);
    \draw (5.5,-1) -- (6,-1);
    \draw (5.5,0) -- (6,0);
    \draw[pattern color=yellow, pattern = north east lines] (0,-1) -- (1,-1) -- (1,5) -- (0,5) -- (0,-1) --
     cycle;
     \draw[pattern color=yellow, pattern = north east lines] (0,7) -- (1,7) -- (1,11) -- (0,11) -- (0,7) --
     cycle;
    \draw[pattern color=orange, pattern = north east lines] (1,-1) -- (2,-1) -- (2,3) -- (1,3) -- (1,-1) --
     cycle;
    \draw[pattern color=green, pattern = north east lines] (2,-1) -- (3,-1) -- (3,2) -- (2,2) -- (2,-1) --
     cycle;
    \draw[pattern color=red, pattern = north east lines] (3,-1) -- (4,-1) -- (4,1) -- (3,1) -- (3,-1) --
     cycle;
   \draw[pattern color=brown, pattern = north east lines] (6,-1) -- (7,-1) -- (7,0) -- (6,0) -- (6,-1) --
     cycle;
\end{tikzpicture}
\caption{}\label{fig:G5cylinder(0,1)}
\end{subfigure}%
\hfill
\begin{subfigure}[b]{0.5\textwidth}
\centering
\begin{tikzpicture}[scale=0.6]
	\draw (0,-1) rectangle (4,0);
	\draw (6,-1) rectangle (7,0);
	\draw (0,-1) rectangle (3,2);
	\draw (0,-1) rectangle (2,3);
	\draw (0,-1) rectangle (1,5);
	\draw (0,-1) rectangle (4,1);
	\draw (0,7) rectangle (1,11);
	\draw (1,0) -- (1,3);
	\draw (2,0) -- (2,2);
	\draw (3,0) -- (3,1);
	\draw (0,4) -- (1,4);
	\draw (0,7) -- (1,7);
	\draw (0,8) -- (1,8);
	\draw (0,9) -- (1,9);
	\draw (0,10) -- (1,10);
    \draw [decorate,line width=0.5mm,decoration={brace}] (-0.5,-1) --  (-0.5,11) node[pos=0.5,left=10pt,black,rotate=-90]{$M=6+4m$};
    \draw [decorate,line width=0.5mm,decoration={brace,mirror}] (0,-2.1) --  (7,-2.1) node[pos=0.5,below=10pt,black]{$N$};
    \draw [dashed] (4.8,-0.5) -- (5.2,-0.5);
    \draw [dashed] (0.5,5.8) -- (0.5,6.2);
    \draw [decorate, decoration={snake}] (0,5.5) -- (1,5.5);
    \draw [decorate, decoration={snake}] (0,6.5) -- (1,6.5);
    \draw (0,5) -- (0,5.5);
    \draw (1,5) -- (1,5.5);
    \draw (0,6.5) -- (0,7);
    \draw (1,6.5) -- (1,7);
    \draw [decorate, decoration={snake}] (4.5,-1) -- (4.5,0);
    \draw [decorate, decoration={snake}] (5.5,-1) -- (5.5,0);
    \draw (4,-1) -- (4.5,-1);
    \draw (4,0) -- (4.5,0);
    \draw (5.5,-1) -- (6,-1);
    \draw (5.5,0) -- (6,0);
    \draw[pattern color=brown, pattern = north east lines] (0,-1) -- (4,-1) -- (4,0) -- (0,0) -- (0,-1) -- cycle;
     \draw[pattern color=brown, pattern = north east lines] (6,-1) -- (7,-1) -- (7,0) -- (6,0) -- (6,-1) -- cycle;
    \draw[pattern color=yellow, pattern = north east lines] (0,0) -- (4,0) -- (4,1) -- (0,1) -- (0,0) -- cycle;
   \draw[pattern color=red, pattern = north east lines] (0,1) -- (3,1) -- (3,2) -- (0,2) -- (0,1) -- cycle;  
   \draw[pattern color=green, pattern = north east lines] (0,2) -- (2,2) -- (2,3) -- (0,3) -- (0,2) -- cycle; 
   \draw[pattern color=orange, pattern = north east lines] (0,3) -- (1,3) -- (1,5) -- (0,5) -- (0,3) -- cycle; 
   \draw[pattern color=orange, pattern = north east lines] (0,7) -- (1,7) -- (1,11) -- (0,11) -- (0,7) -- cycle;
\end{tikzpicture}	
\caption{}\label{fig:G5cylinder(1,0)}
\end{subfigure}
\caption{Origami \(\mathcal{O}^{(5)}_{N,M}\) with cylinder decomposition in vertical and horizontal direction.}
\end{figure}
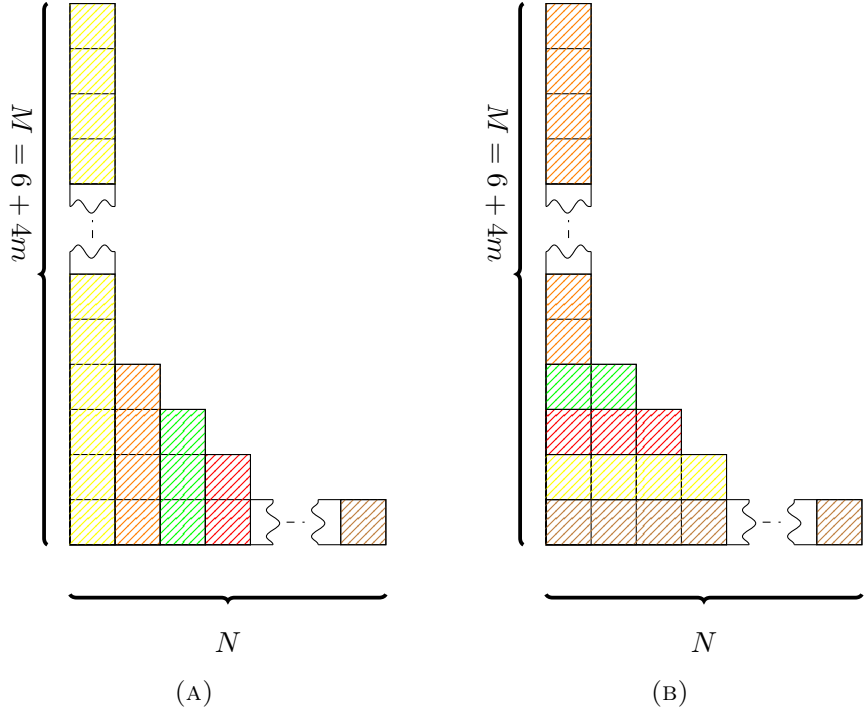


For the cylinder decompositions of the origami \(\mathcal{O}_{N,M}^{(5)}\) in horizontal, resp. vertical direction we get in both cases five maximal cylinders with moduli \(1/(M-4),\ 2/1,\ 3/1,4/1,\ N/1\) for the horizontal direction and moduli \(1/(N-4),\ 2/1,\ 3/1,4/1,\ M/1\) for the vertical direction (see Figure \ref{fig:G5cylinder(0,1)} and Figure \ref{fig:G5cylinder(1,0)}). 
We get two Dehn twists which act on \(H_1^{(0)}(\mathcal{O}_{N,M}^{(5)},\mathbb{Q})\) by the following mapping rules:
\begin{IEEEeqnarray*}{lCCl}
D_h \colon ~  w &\longmapsto& w &+ 12(M-4)N\ \Omega(\sigma_1,w)\,\sigma_1 + 6N\, \Omega(\sigma_2,w)\,\sigma_2  \\ 
                &               &   &  +4N\ \Omega(\sigma_3,w)\,\sigma_3 + 3N\,\Omega(\sigma_4,w)\,\sigma_4 + 12\ \Omega(\sigma_N,w)\,\sigma_N,\\[4mm]
D_v \colon~   w &\longmapsto& w &+ 12(N-4)M\ \Omega(\zeta_1,w)\,\zeta_1 + 6M\ \Omega(\zeta_2,w)\,\zeta_2 \\
                &       &   &+ 4M\ \Omega(\zeta_3,w)\,\zeta_3 +3N\,\Omega(\zeta_4,w)\,\zeta_4+ 12\ \Omega(\zeta_M,w)\,\zeta_M.
\end{IEEEeqnarray*}

It is now easy to calculate representation matrices \( M_h^{(0)}\) and \(M_v^{(0)}\) for the action of the horizontal and vertical twist on \(H_1^{(0)}(\mathcal{O}_{N,M}^{(5)},\mathbb{Q})\) with respect to the basis \(B^{(0)}\):
\begin{displaymath}
M_h^{(0)}=
\left(\begin{array}{cccccccc}
1 & 0 & 0 & 0 & 0    & 0    & 6\,N & 6\,N\\
0 & 1 & 0 & 0 & 0    & 4\,N & 4\,N & 4\,N\\
0 & 0 & 1 & 0 & 3\,N & 3\,N & 3\,N & 3\,N\\
0 & 0 & 0 & 1 & -12  & -24  & -36  & 12-12\,M\\
0 & 0 & 0 & 0 & 1    & 0    & 0    & 0\\
0 & 0 & 0 & 0 & 0    & 1    & 0    & 0\\
0 & 0 & 0 & 0 & 0    & 0    & 1    & 0\\
0 & 0 & 0 & 0 & 0    & 0    & 0    & 1
\end{array}\right)
\end{displaymath}
and
\begin{displaymath}
M_v^{(0)}=
\left(\begin{array}{cccccccc}
1 & 0 & 0 & 0 & 0 & 0 & 0 & 0\\
0 & 1 & 0 & 0 & 0 & 0 & 0 & 0\\
0 & 0 & 1 & 0 & 0 & 0 & 0 & 0\\
0 & 0 & 0 & 1 & 0 & 0 & 0 & 0\\
0 & 0 & -6\,M & -6\,M & 1 & 0 & 0 & 0\\
0 & -4\,M & -4\,M & -4\,M & 0 & 1 & 0 & 0\\
-3\,M & -3\,M & -3\,M & -3\,M & 0 & 0 & 1 & 0\\
12 & 24 & 36 & -12+12\,N & 0 & 0 & 0 & 1
\end{array}\right)
\end{displaymath}

\subsection{Finding a family of candidates in genus five}\label{SubSec:TranUnipG5}

If we compare the length of waist curves obtained by cylinder decompositions in the previous subsection, we can find the following elements of the non-tautological part $H_1^{(0)}(\mathcal{O}_{N,M}^{(5)},\mathbb{Z})$:
\begin{IEEEeqnarray*}{rCl}
    X    &= &   (3m+10)\,\chi_1-(N+m+1)\,\chi_2 \\
         &= &   (N+m+1)\,\Sigma_1+(N+m+1)\,\Sigma_2\\
         &  &+ (N+m+1)\Sigma_3-(3m+10)\Sigma_N \\
         &  &+ (2N-4m-18)\,Z_1+(2N-4m-18)\,Z_2\\
         &  &+ (3N-7)\,Z_3+ (3N-7)\, Z_M,              \\[2mm]
  A      &= &  (9+N+4M)\alpha_1-2\alpha_2                                              \\
         &= & -(9+N+4m)\Sigma_1+2\,\Sigma_2+2\,\Sigma_3+2\Sigma_N\\
         &  & -4\,Z_1-4\,Z_2+(7+N+4m)\,Z_3-4\,Z_M, \\[2mm]
  \Gamma &= &  (8+N+2m)\gamma_1-(3+2m)\gamma_2 \\
         &= & -(N+2m+8)\,\Sigma_1+(2m+3)\,\Sigma_2+(2m+3)\,\Sigma_3+(2m+3)\,\Sigma_N \\
         &  & + M\;Z_1+M\,Z_2+M\,Z_3-(N+5)\,Z_M.  
\end{IEEEeqnarray*}

We consider the subspace $W=\text{Span}_\mathbb{Q}(\{X,A,\Gamma\})$ of $H_1^{(0)}(\mathcal{O}_{N,M}^{(5)},\mathbb{Q})$. With respect to the basis $\{X,A,\Gamma\}$, we can represent the restrictions of the Dehn twists $D_\chi$, $D_\alpha$ and $D_\gamma$ to $W$ by the following three matrices (compare Equations \ref{Eq:DTgamma}, \ref{Eq:DTalpha} and \ref{Eq:DTchi}):
\begin{displaymath}
\begin{pmatrix}
1 & b & a \\
0 & 1 & 0 \\
0 & 0 & 1
\end{pmatrix},
\qquad
\begin{pmatrix}
 1 & 0 & 0 \\
-b & 1 & c \\
 0 & 0 & 1
\end{pmatrix},
\qquad
\begin{pmatrix}
 1 & 0  & 0 \\
 0 & 1  & 0 \\
-a & -c & 1
\end{pmatrix},
\end{displaymath}
where $a=-5N-3Nm+5m+5$, $b=-9N+21$ and $c=-2N+8m+2$. The element $e:=-c\,X+a\,A-b\,\Gamma\in W$ is invariant under $(D_\chi)|_W,~(D_\alpha)|_W$ respectively $(D_\gamma)|_W$. The two waist curves $\gamma_1,\gamma_2$ are linear independent in $H_1(\mathcal{O}_{N,M}^{(5)},\mathbb{Q})$ and the same holds for the waist curves $\alpha_1,\alpha_2$ and $\chi_1,\chi_2$. From the definition of the transvections $D_\chi$, $D_\alpha$ and $D_\gamma$ and the fact that the element $e$ is invariant under them, we can directly see $\Omega(e,w)=0$ for all $w\in W$. With respect to the new basis $\{X,A,e\}$ of $W$ we have the following matrix representations for $(D_\chi)|_W,~(D_\alpha)|_W$ and $(D_\gamma)|_W$:
\begin{displaymath}
\begin{pmatrix}
1 & b & 0 \\
0 & 1 & 0 \\
0 & 0 & 1
\end{pmatrix},
\qquad
\begin{pmatrix}
 1 & 0 & 0 \\
-b & 1 & 0 \\
 0 & 0 & 1
\end{pmatrix},
\qquad
\begin{pmatrix}
 \frac{ac}{b}+1 & \frac{c^2}{b}   & 0 \\
 -\frac{a^2}{b} & -\frac{ac}{b}+1  & 0 \\
\frac{a}{b}     & \frac{c}{b}      & 1
\end{pmatrix} 
\end{displaymath}
whereby $b=-9N+21\neq 0 $ for all $N\in\mathbb{N}$. If we now choose $c=0$ or equivalent $N=1+4m$ one can easily see that $a,b\neq 0 $ for all $N,~m\in\mathbb{N}$ and the subgroup of $\text{Sp}_\Omega(W)$ generated by $(D_\chi)|_W,~(D_\alpha)|_W$ and $(D_\gamma)|_W$ contains a non-trivial element of the unipotent radical.

\subsection{Zariski density and arithmeticity for a genus five family}
In this subsection we fix $M=6+4m$ and $N=1+4m$ ($m\in\mathbb{N}$).
The characteristic polynomial of the matrix $A:=A_5(N,M):=M_h^{(0)}\cdot M_v^{(0)}\in\mathbb{R}^{8\times 8}$ is given by a reciprocal polynomial 
  \begin{IEEEeqnarray*}{lCl}
     P(X)=\chi_A(X)&=& \sum_{i=0}^8 a_i\,X^i\in\mathbb{Z}[X]
  \end{IEEEeqnarray*}
with $a_0=a_8=1$ and $a_i=a_{8-i}$ for $i=1,\dots,4$.

We have $1/X^4\cdot P(X)=Q(X+1/X+2)$ for the quartic polynomial 
  \begin{equation}\label{Eq:CubicQG=5}
      Q(Y)=Y^4+\sum_{i=0}^3 b_i\, Y^i\in\mathbb{Z}[Y]
  \end{equation}
with coefficients
\begin{IEEEeqnarray*}{lClClCl}
    b_3 &=& a_1-8,                  &\quad& b_2 &=& a_2-6\,a_1+20,  \\
    b_1 &=& a_3-4\,a_2+9\,a_1-16,   &\quad& b_0 &=& a_4-2\,a_3+2\,a_2-2\,a_1+2.
\end{IEEEeqnarray*}
Let now $\mu_1,\,\mu_2,\,\mu_3,\,\mu_4\in\mathbb{C}$ be the roots of the polynomial $Q(Y)\in\mathbb{Z}[Y]$. 
and $CR_Q(Y)\in\mathbb{Q}[Y]$ its cubic resolvent (cf. \S \ref{SubSec:GalG45}).
  
For $m\equiv 1$ modulo $31$ the polynomials $P(X)\in\mathbb{Z}[X]$ and $Q(Y)\in\mathbb{Z}[Y]$ are irreducible modulo $31$ and hence irreducible over the rational numbers $\mathbb{Q}$. Furthermore the cubic resolvent $CR_Q(Y)\in\mathbb{Z}[Y]$ is irreducible modulo $11$ if $m\equiv 1$ modulo $11$ and in this case also irreducible over the rational numbers. The discriminant $\text{Disc}(Q)$ of the polynomial $Q(Y)$ can be written by irreducible factors in terms of $m$ as
  \[
  \text{Disc}(Q)=c\cdot f(m)\cdot (2\,m+3)^6\cdot (4\,m+1)^6
  \]
for a positive integer $c$ and a monic polynomial $f(m)$ of degree $12$. With Siegel's theorem of integral points we conclude that $\text{Disc}(Q)$ can only be a square of a rational number for finitely many $m\in\mathbb{N}$. This implies that the Galois group $\text{Gal}(Q)$ of $Q(Y)$ can be identified with the full symmetric group $\text{Sym}(\{\mu_1,\dots,\mu_4\})$ for all but perhaps finitely many $m\in\mathbb{N}$ with $m\equiv 1$ modulo $p\in\{11,\,31\}$. In these cases the Galois group $\text{Gal}(P)\leq \mathbb{Z}_2^4\rtimes S_4$ of our reciprocal polynomial $P(X)=\text{char}_A(X)$ projects surjectively on $S_4$ and hence $\text{Gal}(P)$ can be identified with $S_4$, with one of the groups $H_{4,i}~(i=1,2,3)$ or with the full hyperoctahedral group $G_4=\mathbb{Z}_2^4\rtimes S_4$. 

We can write $\Delta_{4,1}=\delta_{4,1}^2$ from Lemma \ref{Lem:MaxSubDiscr} by irreducible factors in terms of $m$ as
  \[
  \Delta_{4,1}= c\cdot g(m)\cdot (2\,m+1)(4\,m-3)(2\,m+3)^3(4\,m+1)^3
  \]
for an integer $c$ and a monic polynomial $g(m)$ of degree $8$. Furthermore $\Delta_{4,2}=\text{Disc}(Q)\cdot\Delta_{4,1}$. 

Siegel's theorem on integral points implies again that $\delta_{4,1}$ and $\delta_{4,2}$ are rational numbers only for finitely many $m\in\mathbb{N}$. Since $H_{4,3}$ is a subgroup of $H_{4,1}$ the fundamental theorem of Galois theory together with Lemma \ref{Lem:MaxSubDiscr} shows that $\text{Gal}(P)\neq H_{4,i}$ for $i=1,2,3$.

We showed the hardest part of the following proposition:

\begin{proposition}\label{Prop:GalPinG5}
The matrix $A=A_5(N,M)\in\mathbb{R}^{8\times 8}$ is Galois pinching for all but finitely many $m\in\mathbb{N}$ with $m\equiv 1$ modulo $p$, where $p\in\{11,~31\}$. 
\end{proposition}
\begin{proof}
    The only thing that is left to show, is that for $m\in\mathbb{N}$ as in the statement and big enough, the matrix $A_5(N,M)$ has only real eigenvalues. For this reason we analyse the roots of the polynomial $Q(Y)\in\mathbb{Z}$ from Equation (\ref{Eq:CubicQG=5}) with the help of (\ref{Eq:DisRealRoots}) in Section \ref{Sec:RealRoots}.
    First we bring $Q(Y)$ in the depresses form $DQ(t)$ by substituting $t=Y-b_3/4$. If we determine the roots of the depressed form $DQ(t)$, we can determine them directly for $Q(Y)$ as well. By a boring but not to complicated analysis of the expressions $\text{Disc}(DQ)$, $F(DQ)$ and $q$ for $DQ(Y)$ as in \ref{Eq:DisRealRoots} (or by using a computer algebra system), we see that
    \begin{align*}
       \text{Disc}(DQ)>0,\quad F(DQ)<0 \quad{and}\quad q<0 
    \end{align*}
    for $m$ big enough.
    
    Hence in that case all the roots of $DQ(t)$ and thus all the roots of $Q(Y)$ are real. As in the proof of \ref{Prop:GalPinG4} we conclude that $A_5(N,M)$ has only real eigenvalues for $m\in\mathbb{N}$ with $m\equiv 1$ modulo $p$, where $p\in\{11,~31\}$ and $m$ big enough.
\end{proof}

Since $B_5(N,M):=M_\alpha^{(0)}$ is an unipotent matrix such that the image  $(B_5(N,M)-\textrm{Id})(\mathbb{R}^8)$ is not a Lagrangian subspace, the matrices $A_5(N,M)$ and $B_5(N,M)$ generate a Zariski-dense subgroup of $\text{Sp}(H_1^{(0)}(\mathcal{O}_{N,M}^{(5)},\mathbb{Z}))$ for all $m\in\mathbb{N}$ as in Proposition \ref{Prop:GalPinG5}. Together with the theorem of Singh-Venkataramana and the result in Subsection \ref{SubSec:TranUnipG5}, we conclude:

\begin{theorem}\label{t.A.g5}
The genus five origamis $\mathcal{O}^{(5)}_{N,M}\in\mathcal{H}(8)$ with $N=1+4\,m$ and $M=6+4\,m$ have Kontsevich--Zorich monodromies with finite index in $\text{Sp}(H_1^{(0)}(\mathcal{O}_{N,M}^{(5)},\mathbb{Z}))$ for all but finitely many $m\in\mathbb{N}$ such that $m\equiv 1$ modulo $p$, where $p\in\{11,\,31\}$.
\end{theorem}

\section{Genus six stairs}\label{s.g6}

\begin{figure}
\begin{center}

\begin{tikzpicture}[scale=0.6]
	\draw (0,-1) rectangle (8,0);
	\draw (0,-1) rectangle (5,1);
	\draw (0,1) rectangle (4,2);
	\draw (0,2) rectangle (3,3);
	\draw (0,3) rectangle (2,4);
	\draw (0,0) rectangle (1,10);
	\draw (1,-1) -- (1,4);
	\draw (2,-1) -- (2,3);
	\draw (3,-1) -- (3,2);
	\draw (4,-1) -- (4,1);
	\draw (7,-1) -- (7,0);
	\draw (0,5) -- (1,5);
	\draw (0,7) -- (1,7);
	\draw (0,8) -- (1,8);
	\draw (0,9) -- (1,9);
		
    
    \draw [decorate,line width=0.5mm,decoration={brace}] (-0.5,5.1) --  (-0.5,10) node[pos=0.5,left=10pt,black]{$4m$};
    \draw [decorate,line width=0.5mm,decoration={brace,mirror}] (0,-2.1) --  (8,-2.1) node[pos=0.5,below=10pt,black]{$N$};
    
     \draw[thick,color=green] (0,-0.7)--(5.4,-0.7);
     \draw[dashed,thick,color=green] (5.5,-0.7) -- (6.5,-0.7);
	\draw[thick,color=green, ->] (6.6,-0.7)--(8,-0.7);
	\node[color=green, right] (sigmaN) at (8,-0.7) {\(\sigma_N\)};
	\draw[thick,color=green, ->] (0,0.3)--(5,0.3);
	\node[color=green, right] (sigma5) at (5,0.3) {\(\sigma_5\)};
	\draw[thick,color=green, ->] (0,1.3)--(4,1.3);
	\node[color=green, right] (sigma4) at (4,1.3) {\(\sigma_4\)};
	\draw[thick,color=green, ->] (0,2.3)--(3,2.3);
	\node[color=green, right] (sigma3) at (3,2.3) {\(\sigma_3\)};
	\draw[thick,color=green, ->] (0,3.3)--(2,3.3);
	\node[color=green, right] (sigma2) at (2,3.3) {\(\sigma_2\)};
	\draw[thick,color=green, ->] (0,4.3)--(1,4.3);
	\node[color=green, right] (sigma1) at (1,4.3) {\(\sigma_1\)};
	
	\draw[thick,color=red] (0.3,-1)--(0.3,5.4);
	\draw[dashed,thick,color=red] (0.3,5.5) -- (0.3,6.5);
	\draw[thick, color=red,->] (0.3,6.6) -- (0.3,10);
	\node[color=red, below] (zetaM) at (0.3,-1) {\(\zeta_M\)};
	\draw[thick,color=red, ->] (1.3,-1)--(1.3,4);
	\node[color=red, below] (zeta5) at (1.3,-1) {\(\zeta_5\)};
	\draw[thick,color=red, ->] (2.3,-1)--(2.3,3);
	\node[color=red, below] (zeta4) at (2.3,-1) {\(\zeta_4\)};
	\draw[thick,color=red, ->] (3.3,-1)--(3.3,2);
	\node[color=red, below] (zeta3) at (3.3,-1) {\(\zeta_3\)};
	\draw[thick,color=red, ->] (4.3,-1)--(4.3,1);
	\node[color=red, below] (zeta2) at (4.3,-1) {\(\zeta_2\)};
	\draw[thick,color=red, ->] (7.3,-1)--(7.3,0);
    \node[color=red,below] (zeta1) at (7.3,-1) {\(\zeta_1\)};
	
\end{tikzpicture}

\end{center}

\caption{Origami \(\mathcal{O}_{N,M}^{(6)}\) with horizontal waist curves \(\sigma_1,\sigma_2,\sigma_3,\sigma_4,\sigma_5,\sigma_N\) and vertical waist curves \(\zeta_1,\zeta_2,\zeta_3,\zeta_4,\zeta_5,\zeta_M\).}
\label{fig:OrigamiNM6}
\end{figure}
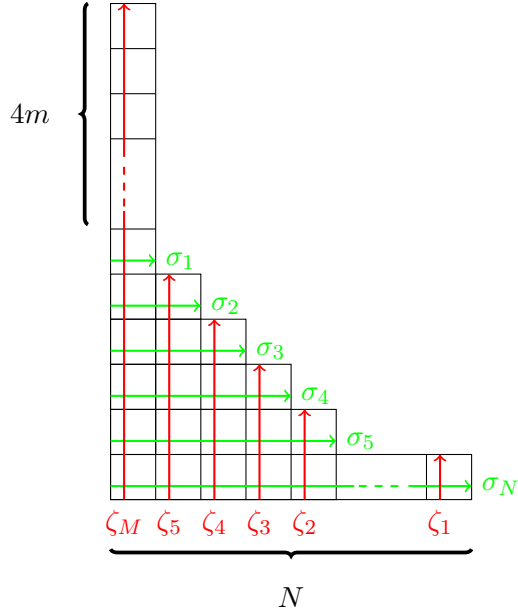

For $N,M\in\mathbb{N}$ with $M=6+4m~(m\in\mathbb{N})$, we consider the origami $\mathcal{O}^{(6)}_{N,M}$ that is given by the following horizontal and vertical permutation $h,v\in\text{Sym}(\{1,2,\dots, N+M+9\})$:
\begin{align*}
  h = &  (1,\dots, N)(N+1,\dots,N+5)(N+6,\dots,N+9)\\
      &  (N+10,N+11,N+12)\\ 
      &  (N+13,N+14)(N+15)\dots(N+M+9)\\[2mm]
  v = &  (1,N+1,N+6,N+10,N+13,N+15,\dots,N+M+9)\\
      &  (2,N+2,N+7,N+11,N+14)(3,N+3,N+8,N+12) \\
      &  (4,N+4,N+12)(5,N+5)(6)\dots(N)
\end{align*}
The six waist curves $\sigma_1,\dots,\sigma_5,~\sigma_N$ of the maximal horizontal cylinders together with the waist curves $\zeta_1,\dots,\zeta_5,\zeta_M$ of the maximal vertical cylinders form again a basis of the absolute homology $H_1(\mathcal{O}^{(6)}_{N,M},\mathbb{Z})$ of the origami $\mathcal{O}^{(6)}_{N,M}$ see Figure \ref{fig:OrigamiNM6}.

It is easy to see that $B^{(0)}=\{\Sigma_i,\,\Sigma_N,\,Z_i,\,Z_M\mid i=1,\dots,4\}$ is a basis of the non-tautological part $H_1^{(0)}(\mathcal{O}_{N,M}^{(6)},\mathbb{Q})$, where
\begin{IEEEeqnarray*}{rClClCrCl}
\Sigma_i&:=&\sigma_{i+1}-(i+1)\,\sigma_1 &~&\mbox{for $i=1,\dots,4$,} &\quad& \Sigma_N &:=& \sigma_N-N\,\sigma_1,\\ 
     Z_i&:=&\zeta_{i+1}-(i+1)\,\zeta_1   &~&\mbox{for $i=1,\dots,4$,} &\quad&      Z_N &:=& \zeta_N-N\,\zeta_1.
\end{IEEEeqnarray*}

We can represent the restriction of the intersection form $\Omega$ to the non-tautological part $H_1^{(0)}(\mathcal{O}_{N,M}^{(6)},\mathbb{Q})$ of the absolute homology by the following matrix with respect to the basis $B^{(0)}$ from above:

\begin{displaymath}
\left(\begin{array}{cccccccccc}
0 & 0 & 0 & 0 &     0 & 0 & 0  & 0  & 1 & -1\\
0 & 0 & 0 & 0 &     0 & 0 & 0  & 1  & 1 & -2\\
0 & 0 & 0 & 0 &     0 & 0 & 1  & 1  & 1 & -3\\
0 & 0 & 0 & 0 &     0 & 1 & 1  & 1  & 1 & -4\\
0 & 0 & 0 & 0 &     0 &-1 & -2 & -3 &-4 & 1-N-M\\
0 & 0 & 0 &-1 &     1 & 0 &  0 & 0  & 0 & 0\\
0 & 0 &-1 & -1&     2 & 0 &  0 & 0  & 0 & 0\\
0 &-1 &-1 & -1&     3 & 0 &  0 & 0  & 0 & 0\\
-1&-1 &-1 & -1&     4 & 0 &  0 & 0  & 0 & 0\\
1 & 2 & 3 &  4& M+N-1 & 0 &  0 & 0  & 0 & 0
\end{array}\right)
\end{displaymath}

\subsection{Dehn twists in genus six}\label{SubSec:DTG6}


\begin{figure}[htb]
\center
\begin{tikzpicture}[scale=0.7]
	\draw (0,-1) rectangle (5,0);
	\draw (7,-1) rectangle (8,0);
	\draw (0,-1) rectangle (3,3);
	\draw (0,-1) rectangle (2,4);
	\draw (0,-1) rectangle (1,5);
	\draw (0,-1) rectangle (4,2);
	\draw (0,-1) rectangle (5,1);
	\draw (0,7) rectangle (1,11);
	\draw (1,0) -- (1,3);
	\draw (2,0) -- (2,2);
	\draw (3,0) -- (3,1);
	\draw (0,4) -- (1,4);
	\draw (0,7) -- (1,7);
	\draw (0,8) -- (1,8);
	\draw (0,9) -- (1,9);
	\draw (0,10) -- (1,10);

    \draw [decorate,line width=0.5mm,decoration={brace}] (-0.5,-1) --  (-0.5,11) node[pos=0.5,left=10pt,black]{$M=6+4m$};
    \draw [decorate,line width=0.5mm,decoration={brace,mirror}] (0,-2.1) --  (8,-2.1) node[pos=0.5,below=10pt,black]{$N$};
    
    \draw [dashed] (5.8,-0.5) -- (6.2,-0.5);
    \draw [dashed] (0.5,5.8) -- (0.5,6.2);

    \draw [decorate, decoration={snake}] (0,5.5) -- (1,5.5);
    \draw [decorate, decoration={snake}] (0,6.5) -- (1,6.5);
    \draw (0,5) -- (0,5.5);
    \draw (1,5) -- (1,5.5);
    \draw (0,6.5) -- (0,7);
    \draw (1,6.5) -- (1,7);
    \draw [decorate, decoration={snake}] (5.5,-1) -- (5.5,0);
    \draw [decorate, decoration={snake}] (6.5,-1) -- (6.5,0);
    \draw (5,-1) -- (5.5,-1);
    \draw (5,0) -- (5.5,0);
    \draw (6.5,-1) -- (7,-1);
    \draw (6.5,0) -- (7,0);


\draw[pattern color=blue, pattern = north east lines] (0,1) -- (1,-1) -- (1.5,-1) -- (0,2) -- (0,1)    --cycle;
\draw[pattern color=blue, pattern = north east lines] (1,4) -- (3.5,-1) -- (4,-1) -- (1.5,4) -- (1,4) --cycle;
\draw[pattern color=blue, pattern = north east lines] (3.5,2) -- (4,1) -- (4,2) -- (3.5,2) --cycle;


\draw[thick, color=red] (0.5,11) -- (1,10);
\draw[thick, color=red] (0,11) -- (1,9);
\draw[thick, color=red] (0,10) -- (1,8);
\draw[thick, color=red] (0,9) -- (1,7);
\draw[thick, color=red] (0,8) -- (0.5,7);
\draw[thick, color=red] (0,5) -- (3,-1);
\draw[thick, color=red] (0,4) -- (2.5,-1);
\draw[thick, color=red] (0,3) -- (2,-1);
\draw[thick, color=red] (0,2) -- (1.5,-1);
\draw[thick, color=red] (0,1) -- (1,-1);
\draw[thick, color=red] (0,0) -- (0.5,-1);
\draw[thick, color=red] (0.5,5) -- (3.5,-1);
\draw[thick, color=red] (1.5,4) -- (4,-1);
\draw[thick, color=red] (2.5,3) -- (4.5,-1);
\draw[thick, color=red] (3.5,2) -- (5,-1);
\draw[thick, color=red] (4.5,1) -- (5,0);
\draw[thick, color=red] (7,0) -- (7.5,-1);
\draw[thick, color=red] (7.5,0) -- (8,-1);

\end{tikzpicture}
\caption{Origami \(\mathcal{O}^{(6)}_{N,M}\) with cylinder decomposition in direction \((1,-2)\). Here $\gamma_1$ is the waist curve of the blue cylinder.}
\label{fig:G6cylinder(1,-2)}
\end{figure}
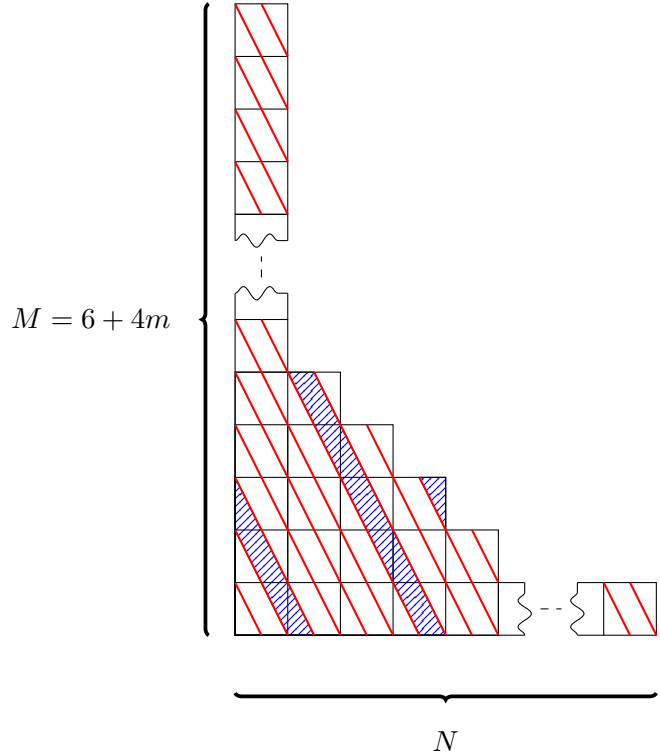


In direction $(1,-2)$ there are two maximal cylinders with one waist curve $\gamma_1$ of combinatorial length $4$ and one waist curve $\gamma_2$ of combinatorial length $N+4m+11$. We count intersection points of $\gamma_1$ and $\gamma_2$ with the elements of $B^{(0)}$ and get
\begin{IEEEeqnarray*}{lClClClClCl}
\Omega(\gamma_1,\sigma_1)&=&0, &\quad& \Omega(\gamma_1,\sigma_2)&=&1, &\quad& \Omega(\gamma_1,\sigma_3)&=&1 , \\
\Omega(\gamma_1,\sigma_4)&=&2, &     & \Omega(\gamma_1,\sigma_5)&=&2, &     & \Omega(\gamma_1,\sigma_N)&=&2 , \\[4mm]
\Omega(\gamma_1,\zeta_1) &=&0, &     & \Omega(\gamma_1,\zeta_2) &=&0, &     & \Omega(\gamma_1,\zeta_3) &=&1  , \\
\Omega(\gamma_1,\zeta_4) &=&1, &     & \Omega(\gamma_1,\zeta_5) &=&1, &     & \Omega(\gamma_1,\zeta_M) &=&1  , \\[4mm]
\Omega(\gamma_2,\sigma_1)&=&2, &     & \Omega(\gamma_2,\sigma_2)&=&3, &     & \Omega(\gamma_2,\sigma_3)&=&5 , \\
\Omega(\gamma_2,\sigma_4)&=&6, &     & \Omega(\gamma_2,\sigma_5)&=&8, &\quad& \Omega(\gamma_2,\sigma_N)&=&2(N-1), \\[4mm]
\Omega(\gamma_2,\zeta_1) &=&1, &     & \Omega(\gamma_2,\zeta_2) &=&2, &     & \Omega(\gamma_2,\zeta_3) &=&2, \\
\Omega(\gamma_2,\zeta_4) &=&3, &     & \Omega(\gamma_2,\zeta_5) &=&4, &     &\Omega(\gamma_2,\zeta_M)  &=&5+4m. \\
\end{IEEEeqnarray*}
We can write $\Gamma:=(N+4m+11)\gamma_1-4\gamma_2\in H_1^{(0)}(\mathcal{O}_{N,M}^{(6)},\mathbb{Z})$ as a linear combination of elements of $B^{(0)}$ in the following way:
\begin{align*}
\Gamma = & \quad 4\,\Sigma_1 +4\,\Sigma_2-(N+4m+11)\,\Sigma_3+4\,\Sigma_4 +4\,\Sigma_N \\
         &- 8\,Z_1+(N+4m+7)\,Z_2-8\,Z_3+(N+4m+7)\,Z_4-8\,Z_M.
\end{align*}
The Dehn twist along the waist curves $\gamma_1$ and $\gamma_2$ acts on $H_1^{(0)}(\mathcal{O}_{N,M}^{(6)},\mathbb{Q})$ via the mapping
\begin{displaymath}
  D_\gamma\colon v\longmapsto v + (N+4m+11)\,\Omega(\gamma_1,v)\,\gamma_1 + 4\,\Omega(\gamma_2,v)\,\gamma_2
\end{displaymath}
and for the images of the elements in $B^{(0)}$ under $D_\gamma$ we get
\begin{IEEEeqnarray*}{lClClClClCl}
D_\gamma(\Sigma_1)&=&\Sigma_1 +\Gamma,    &\quad& D_\gamma(\Sigma_2) &=&\Sigma_2 +\Gamma,   &\quad& D_\gamma(\Sigma_3)&=&\Sigma_3 +2\,\Gamma, \\
D_\gamma(\Sigma_4)&=&\Sigma_4+ 2\,\Gamma, &\quad& D_\gamma(\Sigma_N) &=&\Sigma_N+2\,\Gamma, &\quad &   && \\
D_\gamma(Z_1)     &=& Z_1,                &\quad & D_\gamma(Z_2)     &=&Z_2+\Gamma,         &\quad& D_\gamma(Z_3)&=&Z_3+\Gamma \\
D_\gamma(Z_4)     &=& Z_4 +\Gamma,        &\quad& D_\gamma(Z_M)      &=&Z_M+\Gamma.         &\quad&    &&
\end{IEEEeqnarray*}


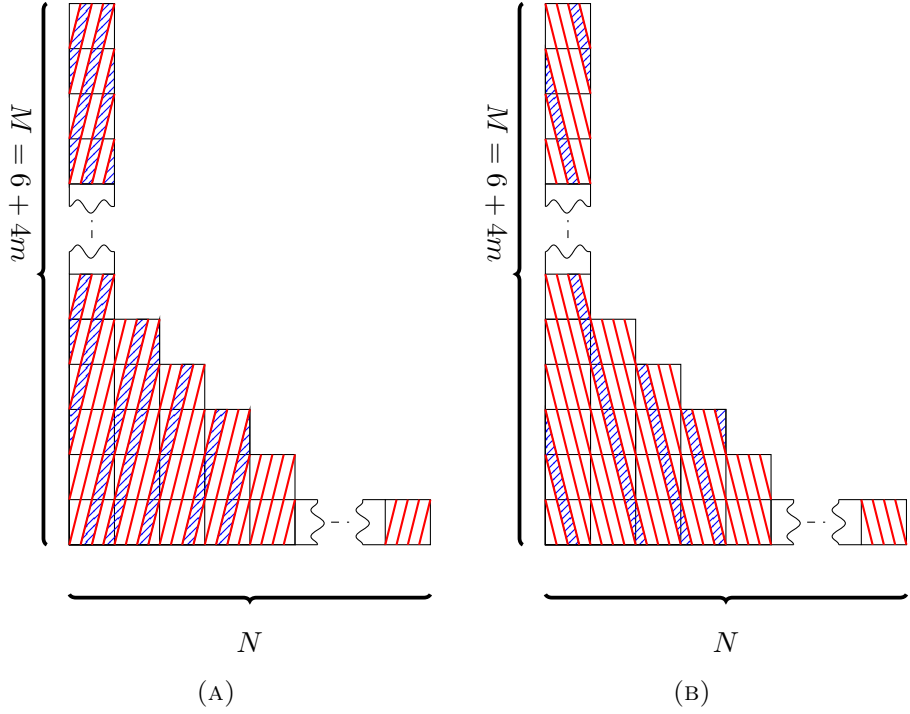
\begin{figure}
\centering
\noindent
\begin{subfigure}[b]{0.5\textwidth}
\centering
\begin{tikzpicture}[scale=0.6]
	\draw (0,-1) rectangle (5,0);
	\draw (7,-1) rectangle (8,0);
	\draw (0,-1) rectangle (3,3);
	\draw (0,-1) rectangle (2,4);
	\draw (0,-1) rectangle (1,5);
	\draw (0,-1) rectangle (4,2);
	\draw (0,-1) rectangle (5,1);
	\draw (0,7) rectangle (1,11);
	\draw (1,0) -- (1,3);
	\draw (2,0) -- (2,2);
	\draw (3,0) -- (3,1);
	\draw (0,4) -- (1,4);
	\draw (0,7) -- (1,7);
	\draw (0,8) -- (1,8);
	\draw (0,9) -- (1,9);
	\draw (0,10) -- (1,10);

    \draw [decorate,line width=0.5mm,decoration={brace}] (-0.5,-1) --  (-0.5,11) node[pos=0.5,left=10pt,black,rotate=-90]{$M=6+4m$};
    \draw [decorate,line width=0.5mm,decoration={brace,mirror}] (0,-2.1) --  (8,-2.1) node[pos=0.5,below=10pt,black]{$N$};
    
    \draw [dashed] (5.8,-0.5) -- (6.2,-0.5);
    \draw [dashed] (0.5,5.8) -- (0.5,6.2);

    \draw [decorate, decoration={snake}] (0,5.5) -- (1,5.5);
    \draw [decorate, decoration={snake}] (0,6.5) -- (1,6.5);
    \draw (0,5) -- (0,5.5);
    \draw (1,5) -- (1,5.5);
    \draw (0,6.5) -- (0,7);
    \draw (1,6.5) -- (1,7);
    \draw [decorate, decoration={snake}] (5.5,-1) -- (5.5,0);
    \draw [decorate, decoration={snake}] (6.5,-1) -- (6.5,0);
    \draw (5,-1) -- (5.5,-1);
    \draw (5,0) -- (5.5,0);
    \draw (6.5,-1) -- (7,-1);
    \draw (6.5,0) -- (7,0);

   \draw[pattern color=blue, pattern = north east lines] (0,10) -- (0,9) -- (0.5,11) -- (0.25,11) -- (0,10) --cycle;
   \draw[pattern color=blue, pattern = north east lines] (0,7) -- (1,11) -- (0.75,11) -- (0,8) -- (0,7) --cycle;
   \draw[pattern color=blue, pattern = north east lines] (0.25,7) -- (0.5,7) -- (1,9) -- (1,10) -- (0.25,7) --cycle;
   \draw[pattern color=blue, pattern = north east lines] (0.75,7) -- (1,7) -- (1,8) -- (0.75,7) -- cycle;
   \draw[pattern color=blue, pattern = north east lines] (0,3) -- (0.5,5) -- (0.25,5) -- (0,4) -- (0,3) --cycle;
   \draw[pattern color=blue, pattern = north east lines] (0,1) -- (1,5) -- (0.75,5) -- (0,2) -- (0,1) --cycle;
   \draw[pattern color=blue, pattern = north east lines] (0.25,-1) -- (0.5,-1) -- (1.75,4) -- (1.5,4) -- (0.25,-1) --cycle;
   \draw[pattern color=blue, pattern = north east lines] (0.75,-1) -- (1,-1) -- (2,3) -- (2,4) -- (0.75,-1) --cycle;
   \draw[pattern color=blue, pattern = north east lines] (1.5,-1) -- (1.75,-1) -- (2.75,3) -- (2.5,3) -- (1.5,-1) --cycle;
   \draw[pattern color=blue, pattern = north east lines] (2.5,-1) -- (2.75,-1) -- (3.5,2) -- (3.25,2) -- (2.5,-1) --cycle;
   \draw[pattern color=blue, pattern = north east lines] (3.25,-1) -- (3.5,-1) -- (4,1) -- (4,2) -- (3.25,-1) --cycle;

   \draw[thick, color=red] (0,10) -- (0.25,11);
   \draw[thick, color=red] (0,9) -- (0.5,11);
   \draw[thick, color=red] (0,8) -- (0.75,11);
   \draw[thick, color=red] (0,7) -- (1,11);
   \draw[thick, color=red] (0.25,7) -- (1,10);
   \draw[thick, color=red] (0.5,7) -- (1,9);
   \draw[thick, color=red] (0.75,7) -- (1,8);
   \draw[thick, color=red] (0,4) -- (0.25,5);
   \draw[thick, color=red] (0,3) -- (0.5,5);
   \draw[thick, color=red] (0,2) -- (0.75,5);
   \draw[thick, color=red] (0,1) -- (1,5);
   \draw[thick, color=red] (0,0) -- (1,4);
   \draw[thick, color=red] (0,-1) -- (1.25,4);
   \draw[thick, color=red] (0.25,-1) -- (1.5,4);
   \draw[thick, color=red] (0.5,-1) -- (1.75,4);
   \draw[thick, color=red] (0.75,-1) -- (2,4);
   \draw[thick, color=red] (1,-1) -- (2,3);
   \draw[thick, color=red] (1.25,-1) -- (2.25,3);
   \draw[thick, color=red] (1.5,-1) -- (2.5,3);
   \draw[thick, color=red] (1.75,-1) -- (2.75,3);
   \draw[thick, color=red] (2,-1) -- (3,3);
   \draw[thick, color=red] (2.25,-1) -- (3,2);
   \draw[thick, color=red] (2.5,-1) -- (3.25,2);
   \draw[thick, color=red] (2.75,-1) -- (3.5,2);
   \draw[thick, color=red] (3,-1) -- (3.75,2);
   \draw[thick, color=red] (3.25,-1) -- (4,2);
   \draw[thick, color=red] (3.5,-1) -- (4,1);
   \draw[thick, color=red] (3.75,-1) -- (4.25,1);
   \draw[thick, color=red] (4,-1) -- (4.5,1);
   \draw[thick, color=red] (4.25,-1) -- (4.75,1);
   \draw[thick, color=red] (4.5,-1) -- (5,1);
   \draw[thick, color=red] (4.75,-1) -- (5,0);
   \draw[thick, color=red] (7,-1) -- (7.25,0);
   \draw[thick, color=red] (7.25,-1) -- (7.5,0);
   \draw[thick, color=red] (7.5,-1) -- (7.75,0);
   \draw[thick, color=red] (7.75,-1) -- (8,0);
   
\end{tikzpicture}	
\caption{}\label{fig:G6cylinder(1,4)}
\end{subfigure}%
\hfill
\begin{subfigure}[b]{0.5\textwidth}
\centering
\begin{tikzpicture}[scale=0.6]
	\draw (0,-1) rectangle (5,0);
	\draw (7,-1) rectangle (8,0);
	\draw (0,-1) rectangle (3,3);
	\draw (0,-1) rectangle (2,4);
	\draw (0,-1) rectangle (1,5);
	\draw (0,-1) rectangle (4,2);
	\draw (0,-1) rectangle (5,1);
	\draw (0,7) rectangle (1,11);
	\draw (1,0) -- (1,3);
	\draw (2,0) -- (2,2);
	\draw (3,0) -- (3,1);
	\draw (0,4) -- (1,4);
	\draw (0,7) -- (1,7);
	\draw (0,8) -- (1,8);
	\draw (0,9) -- (1,9);
	\draw (0,10) -- (1,10);

    \draw [decorate,line width=0.5mm,decoration={brace}] (-0.5,-1) --  (-0.5,11) node[pos=0.5,left=10pt,black,rotate=-90]{$M=6+4m$};
    \draw [decorate,line width=0.5mm,decoration={brace,mirror}] (0,-2.1) --  (8,-2.1) node[pos=0.5,below=10pt,black]{$N$};
    
    \draw [dashed] (5.8,-0.5) -- (6.2,-0.5);
    \draw [dashed] (0.5,5.8) -- (0.5,6.2);

    \draw [decorate, decoration={snake}] (0,5.5) -- (1,5.5);
    \draw [decorate, decoration={snake}] (0,6.5) -- (1,6.5);
    \draw (0,5) -- (0,5.5);
    \draw (1,5) -- (1,5.5);
    \draw (0,6.5) -- (0,7);
    \draw (1,6.5) -- (1,7);
    \draw [decorate, decoration={snake}] (5.5,-1) -- (5.5,0);
    \draw [decorate, decoration={snake}] (6.5,-1) -- (6.5,0);
    \draw (5,-1) -- (5.5,-1);
    \draw (5,0) -- (5.5,0);
    \draw (6.5,-1) -- (7,-1);
    \draw (6.5,0) -- (7,0);


   \draw[pattern color=blue, pattern = north east lines] (0.5,11) -- (1,9) -- (1,10) -- (0.75,11) -- (0.5,11) --cycle;
   \draw[pattern color=blue, pattern = north east lines] (0,10) -- (0.75,7) -- (0.5,7) -- (0,9) -- (0,10) --cycle;
   \draw[pattern color=blue, pattern = north east lines] (0,1) -- (0.5,-1) -- (0.75,-1) -- (0,2) -- (0,1) --cycle;
   \draw[pattern color=blue, pattern = north east lines] (0.5,5) -- (2,-1) -- (2.25,-1) -- (0.75,5) -- (0.5,5) --cycle;
   \draw[pattern color=blue, pattern = north east lines] (2,3) -- (3,-1) -- (3.25,-1) -- (2.25,3) -- (2,3) --cycle;
   \draw[pattern color=blue, pattern = north east lines] (3,2) -- (3.75,-1) -- (4,-1) -- (3.25,2) -- (3,2) --cycle;
   \draw[pattern color=blue, pattern = north east lines] (3.75,2) -- (4,1) -- (4,2) -- (3.75,2) -- cycle;

   \draw[thick, color=red] (0.75,11) -- (1,10);
   \draw[thick, color=red] (0.5,11) -- (1,9);
   \draw[thick, color=red] (0.25,11) -- (1,8);
   \draw[thick, color=red] (0,11) -- (1,7);
   \draw[thick, color=red] (0,10) -- (0.75,7);
   \draw[thick, color=red] (0,9) -- (0.5,7);
   \draw[thick, color=red] (0,8) -- (0.25,7);
   \draw[thick, color=red] (0.25,-1) -- (0,0);
   \draw[thick, color=red] (0.5,-1) -- (0,1);
   \draw[thick, color=red] (0.75,-1) -- (0,2);
   \draw[thick, color=red] (1,-1) -- (0,3);
   \draw[thick, color=red] (1.25,-1) -- (0,4);
   \draw[thick, color=red] (1.5,-1) -- (0,5);
   \draw[thick, color=red] (1.75,-1) -- (0.25,5);
   \draw[thick, color=red] (2,-1) -- (0.5,5);
   \draw[thick, color=red] (2.25,-1) -- (0.75,5);
   \draw[thick, color=red] (2.5,-1) -- (1.25,4);
   \draw[thick, color=red] (2.75,-1) -- (1.5,4);
   \draw[thick, color=red] (3,-1) -- (1.75,4);
   \draw[thick, color=red] (3.25,-1) -- (2.25,3);
   \draw[thick, color=red] (3.5,-1) -- (2.5,3);
   \draw[thick, color=red] (3.75,-1) -- (2.75,3);
   \draw[thick, color=red] (4,-1) -- (3.25,2);
   \draw[thick, color=red] (4.25,-1) -- (3.5,2);
   \draw[thick, color=red] (4.5,-1) -- (3.75,2);
   \draw[thick, color=red] (4.75,-1) -- (4.25,1);
   \draw[thick, color=red] (5,-1) -- (4.5,1);
   \draw[thick, color=red] (5,0) -- (4.75,1);
   \draw[thick, color=red] (7,0) -- (7.25,-1);
   \draw[thick, color=red] (7.25,0) -- (7.5,-1);
   \draw[thick, color=red] (7.5,0) -- (7.75,-1);
   \draw[thick, color=red] (7.75,0) -- (8,-1);
\end{tikzpicture}	
\caption{}\label{fig:G6cylinder(1,-4)}
\end{subfigure}
\caption{Origami \(\mathcal{O}^{(6)}_{N,M}\) with cylinder decomposition in direction $(1,4)$ and \((1,-4)\). Here $\delta_1$ and $\alpha_1$ are the waist curves of the blue cylinders.}
\end{figure}


For direction $(1,4)$ there are two maximal cylinders with waist curve $\delta_1$ of length $2m+6$ and waist curve $\delta_2$ of length $N+2m+9$. We have
\begin{IEEEeqnarray*}{lClClClClCl}
\Omega(\delta_1,\sigma_1)&=&-2, &\quad& \Omega(\delta_1,\sigma_2)&=&-3, &\quad& \Omega(\delta_1,\sigma_3)&=&-4, \\
\Omega(\delta_1,\sigma_4)&=&-5, &     & \Omega(\delta_1,\sigma_5)&=&-5, &     & \Omega(\delta_1,\sigma_N)&=&-5, \\[2mm]
\Omega(\delta_1,\zeta_1) &=&0,  &     & \Omega(\delta_1,\zeta_2) &=&0,  &     & \Omega(\delta_1,\zeta_3) &=&1, \\
\Omega(\delta_1,\zeta_4) &=&1 , &\quad& \Omega(\delta_1,\zeta_5) &=&2,  &     & \Omega(\delta_1,\zeta_M) &=&2+2m, \\[2mm]
\Omega(\delta_2,\sigma_1)&=&-2, &     & \Omega(\delta_2,\sigma_2)&=&-5, &     & \Omega(\delta_2,\sigma_3)&=&-8, \\
\Omega(\delta_2,\sigma_4)&=&-11,&     & \Omega(\delta_2,\sigma_5)&=&-15,&     & \Omega(\delta_2,\sigma_N)&=&-4N+5, \\[2mm]
\Omega(\delta_2,\zeta_1) &=&1,  &     & \Omega(\delta_2,\zeta_2) &=&2,  &     & \Omega(\delta_2,\zeta_3) &=&2  , \\
\Omega(\delta_2,\zeta_4) &=&3,  &     & \Omega(\delta_2,\zeta_5) &=&3,  &     & \Omega(\delta_2,\zeta_M) &=&4+2m.
\end{IEEEeqnarray*}
The element $\Delta:=(N+2m+9)\delta_1-(2m+6)\delta_2\in H_1^{(0)}(\mathcal{O}_{N,M}^{(6)},\mathbb{Z})$ can be written as a linear combination of elements of $B^{(0)}$ as
\begin{align*}
\Delta = & -(N+2m+9)\,\Sigma_1+(2m+6)\,\Sigma_2-(N+2m+9)\,\Sigma_3\\
         &+(2m+6)\,\Sigma_4+(2m+6)\,\Sigma_N \\
         & +(8m+24)\,Z_1 +(4m-N+9)\,Z_2 +(4m-N+9)\,Z_3\\
         &+(4m-N+9)\,Z_4 -(2N+6)\,Z_M.
\end{align*}
The Dehn twist along the waist curves $\delta_1$ and $\delta_2$ of the maximal cylinders acts on the non-tautological part of the absolute homology via the mapping
\begin{displaymath}
  D_\delta\colon v\longmapsto v + (N+2m+9)\,\Omega(\delta_1,v)\,\delta_1+(2m+6)\,\Omega(\delta_2,v)\,\delta_2.
\end{displaymath}
If we evaluate the elements of the basis $B^{(0)}$ of $H_1^{(0)}(\mathcal{O}_{N,M}^{(6)},\mathbb{Q})$, then we get
\begin{IEEEeqnarray*}{lClClClClCl}
D_\delta(\Sigma_1)&=&\Sigma_1+\Delta,   &\quad& D_\delta(\Sigma_2)&=&\Sigma_2+2\,\Delta,      &\quad& D_\delta(\Sigma_3)&=&\Sigma_3+3\,\Delta, \\
D_\delta(\Sigma_4)&=&\Sigma_4+5\,\Delta,&\quad& D_\delta(\Sigma_N)&=&\Sigma_n+(2N-5)\,\Delta, &\quad&  && \\
D_\delta(Z_1)     &=& Z_1 ,             &\quad& D_\delta(Z_2)     &=&Z_2+\Delta,              &\quad&  D_\delta(Z_3)&=&Z_3+\Delta, \\
D_\delta(Z_4)     &=& Z_4+2\,\Delta,    &\quad& D_\delta(Z_M)     &=&Z_M+(2+2m)\,\Delta.      &\quad&  &&
\end{IEEEeqnarray*}

We have two maximal cylinders in direction $(1,-4)$ with waist curve $\alpha_1$ of combinatorial length $4+m$ and waist curve $\alpha_2$ of combinatorial length $N+3m+11$. We calculate the following intersection points with the elements of the basis $B^{(0)}$:
\begin{IEEEeqnarray*}{lClClClClCl}
\Omega(\alpha_1,\sigma_1)&=&1, &\quad& \Omega(\alpha_1,\sigma_2)&=&1, &\quad& \Omega(\alpha_1,\sigma_3)&=&2 , \\
\Omega(\alpha_1,\sigma_4)&=&4, &\quad& \Omega(\alpha_1,\sigma_5)&=&4, &\quad& \Omega(\alpha_1,\sigma_N)&=&4 , \\[2mm]
\Omega(\alpha_1,\zeta_1) &=&0, &     & \Omega(\alpha_1,\zeta_2) &=&0, &     & \Omega(\alpha_1,\zeta_3) &=&1 , \\
\Omega(\alpha_1,\zeta_4) &=&1, &     & \Omega(\alpha_1,\zeta_5) &=&1, &     & \Omega(\alpha_1,\zeta_M) &=&1+m , \\[2mm]
\Omega(\alpha_2,\sigma_1)&=&3, &     & \Omega(\alpha_2,\sigma_2)&=&7, &     & \Omega(\alpha_2,\sigma_3)&=&10 , \\
\Omega(\alpha_2,\sigma_4)&=&12,&     &\Omega(\alpha_2,\sigma_5) &=&16,&     & \Omega(\alpha_2,\sigma_N)&=&N-4 , \\[2mm]
\Omega(\alpha_2,\zeta_1) &=&1, &     & \Omega(\alpha_2,\zeta_2) &=&2, &     & \Omega(\alpha_2,\zeta_3) &=&2 , \\
\Omega(\alpha_2,\zeta_4) &=&3, &     & \Omega(\alpha_2,\zeta_5) &=&4, &     & \Omega(\alpha_2,\zeta_M) &=&5+3m.
\end{IEEEeqnarray*}
With this information we can write the element $A:=(N+3m+11)\alpha_1-(4+m)\alpha_2$ in the basis $B^{(0)}$:
\begin{align*}
A= & \quad (m+4)\,\Sigma_1+ (m+4) \, \Sigma_2 -(N+3m+11)\,\Sigma_3\\
   &+(m+4)\,\Sigma_4 +(m+4)\,\Sigma_N \\
   & -(4m+16)\,Z_1+(2N+4m+14)\,Z_2+(N-1)\,Z_3\\
   &-(4m+16)\, Z_4 +(N-1)\,Z_M.
\end{align*}
The map
\begin{displaymath}
D_\alpha\colon v \longmapsto v + (N+3m+11)\,\Omega(\alpha_1,v)\,\alpha_1 +(m+4)\,\Omega(\alpha_2,v)\,\alpha_2
\end{displaymath}
has images
\begin{IEEEeqnarray*}{lClClClClCl}
D_\alpha(\Sigma_1)&=&\Sigma_1-A, &\quad& D_\alpha(\Sigma_2)&=&\Sigma_2-A         &\quad& D_\alpha(\Sigma_3)&=&\Sigma_3    , \\
D_\alpha(\Sigma_4)&=&\Sigma_4-A, &     & D_\alpha(\Sigma_N)&=&\Sigma_N-(N-4)\,A, &     &    && \\
D_\alpha(Z_1)     &=&Z_1,        &     & D_\alpha(Z_2)     &=&Z_2 + A,           &\quad& D_\alpha(Z_3)&=&Z_3 + A, \\
D_\alpha(Z_4)     &=&Z_4+A,      &\quad& D_\alpha(Z_M)     &=&Z_M+(m+1)\, A.     &     &    && 
\end{IEEEeqnarray*}

In horizontal direction we have six maximal cylinders with moduli $1/(M-5)$, $2/1$, $3/1$, $4/1$, $5/1$ and $N/1$ respectively in vertical direction there are six maximal cylinders with moduli $1/(N-5)$, $2/1$, $3/1$, $4/1$, $5/1$ and $N/1$. As in the sections before we can calculate representation matrices for the action of the associated Dehn twists on $H_1^{(0)}(\mathcal{O}_{N,M}^{(6)},\mathbb{Q})$. In horizontal direction we have 
\begin{displaymath}
M_h^{(0)}=
\left(\begin{array}{cccccccccc}
1 & 0 & 0 & 0 & 0 &   0  &  0  &  0  & 30N & 30N \\
0 & 1 & 0 & 0 & 0 &   0  &  0  & 20N & 20N & 20N \\
0 & 0 & 1 & 0 & 0 &   0  & 15N & 15N & 15N & 15N \\
0 & 0 & 0 & 1 & 0 & 12N  & 12N & 12N & 12N & 12N \\
0 & 0 & 0 & 0 & 1 & -60  &-120 &-180 &-240 &-60(M-1)\\
0 & 0 & 0 & 0 & 0 &  1   &  0  &  0  &  0  &  0  \\
0 & 0 & 0 & 0 & 0 &  0   &  1  &  0  &  0  &  0  \\
0 & 0 & 0 & 0 & 0 &  0   &  0  &  1  &  0  &  0  \\
0 & 0 & 0 & 0 & 0 &  0   &  0  &  0  &  1  &  0  \\
0 & 0 & 0 & 0 & 0 &  0   &  0  &  0  &  0  &  1  \\
\end{array}\right)
\end{displaymath}
and in vertical direction we get
\begin{displaymath}
M_v^{(0)}=
\left(\begin{array}{cccccccccc}
1    &  0  &  0  &  0  &  0      & 0 & 0 & 0 & 0 & 0 \\
0    &  1  &  0  &  0  &  0      & 0 & 0 & 0 & 0 & 0 \\
0    &  0  &  1  &  0  &  0      & 0 & 0 & 0 & 0 & 0 \\
0    &  0  &  0  &  1  &  0      & 0 & 0 & 0 & 0 & 0 \\
0    &  0  &  0  &  0  &  1      & 0 & 0 & 0 & 0 & 0 \\
0    &  0  &  0  &-30M &-30M     & 1 & 0 & 0 & 0 & 0 \\
0    &  0  &-20M &-20M &-20M     & 0 & 1 & 0 & 0 & 0 \\
0    &-15M &-15M &-15M &-15M     & 0 & 0 & 1 & 0 & 0 \\
-12M &-12M &-12M &-12M &-12M     & 0 & 0 & 0 & 1 & 0 \\
 60  & 120 & 180 & 240 & 60(N-1) & 0 & 0 & 0 & 0 & 1 
\end{array}\right).
\end{displaymath}

\subsection{Finding a family of candidates in genus six}
Recall that we obtained elements $\Gamma,~\Delta,~A$ of the non-tautological part $H_1^{(0)}(\mathcal{O}_{N,M}^{(6)},\mathbb{Z})$ of the absolute homology by comparing the waist curves of the maximal cylinders in direction $(1,-2)$, $(1,4)$ and $(1,-4)$. We wrote them as a linear combination of elements of 
$B^{(0)}$ in the following way:
\begin{align*}
A=       & \quad (m+4)\,\Sigma_1+ (m+4) \, \Sigma_2 -(N+3m+11)\,\Sigma_3 \\
         &+(m+4)\,\Sigma_4 +(m+4)\,\Sigma_N \\
         & -(4m+16)\,Z_1+(2N+4m+14)\,Z_2+(N-1)\,Z_3\\
         &-(4m+16)\, Z_4 +(N-1)\,Z_M \\[2mm]
\Gamma = & \quad 4\,\Sigma_1 +4\,\Sigma_2-(N+4m+11)\,\Sigma_3\\
         &+4\,\Sigma_4 +4\,\Sigma_N \\
         &- 8\,Z_1+(N+4m+7)\,Z_2-8\,Z_3+(N+4m+7)\,Z_4-8\,Z_M \\[2mm]
\Delta = & -(N+2m+9)\,\Sigma_1+(2m+6)\,\Sigma_2-(N+2m+9)\,\Sigma_3\\
         &+(2m+6)\,\Sigma_4+(2m+6)\,\Sigma_N \\
         & +(8m+24)\,Z_1 +(4m-N+9)\,Z_2 +(4m-N+9)\,Z_3 \\
         &+(4m-N+9)\,Z_4 -(2N+6)\,Z_M.
\end{align*}
Let $W=\text{Span}_\mathbb{Q}(A,\,\Gamma,\,\Delta\}$ the $\mathbb{Q}$-linear subspace of $H_1^{(0)}(\mathcal{O}_{N,M}^{(6)},\mathbb{Q})$ spanned by $A,~\Gamma$ and $\Delta$. The three maps $D_\alpha,~D_\gamma$ and $D_\delta$ are transvections on $H_1^{(0)}(\mathcal{O}_{N,M}^{(6)},\mathbb{Q})$ and if we restrict them to the subspace $W$, they have the following three matrix representations with respect to $\{A,\,\Gamma,\,\Delta\}$
\begin{displaymath}
\begin{pmatrix}
1 & b & a \\
0 & 1 & 0 \\
0 & 0 & 1
\end{pmatrix},
\qquad
\begin{pmatrix}
 1 & 0 & 0 \\
-b & 1 & c \\
 0 & 0 & 1
\end{pmatrix},
\qquad
\begin{pmatrix}
 1 & 0  & 0 \\
 0 & 1  & 0 \\
-a & -c & 1
\end{pmatrix},
\end{displaymath}
where $b=2-2N$, $a=-10\,N+12\,m-4\,mN+42$ and $c=16\,m+24-8\,N$. The element $e:=c\,A-a\,\Gamma+b\,\Delta$ is invariant under the elements $(D_\alpha)|_W,~(D_\gamma)|_W$, $(D_\delta)|_W\in\text{Sp}_\Omega(W)$ and an element of the nullspace $W^\Omega$. The restrictions of $D_\alpha,~D_\gamma$ and $D_\delta$ to the subspace $W$ have the following matrix representations with respect to the basis $\{A,\,\Gamma\,e\}$:
\begin{displaymath}
\begin{pmatrix}
1 & b & 0 \\
0 & 1 & 0 \\
0 & 0 & 1
\end{pmatrix},
\qquad
\begin{pmatrix}
 1 & 0 & 0 \\
-b & 1 & 0 \\
 0 & 0 & 1
\end{pmatrix},
\qquad
\begin{pmatrix}
 \frac{ac}{b}+1  & \frac{c^2}{b}   & 0 \\
 \frac{a^2}{b}   & \frac{ac}{b}+1  & 0 \\
-\frac{a}{b}     & -\frac{c}{b}      & 1
\end{pmatrix}.
\end{displaymath}
If we choose $c=0$ or $N=3+2m$, then we can easily find an element of the unipotent radical of $\text{Sp}_\Omega(W)$ in the subgroup generated by the three transvections $(D_\alpha)|_W,~(D_\gamma)|_W$ and $(D_\delta)|_W$.

\subsection{Zariski density and arithmeticity for a genus six family}
We consider in this subsection the family of origamis $\mathcal{O}_{N,M}^{(6)}$, where $M=6+4m$, $N=3+2m$ and $m\in\mathbb{N}$. As before in the genus four and five section we try first to determine an infinite family of natural numbers$m\in\mathbb{N}$ for which the matrix $A=A_6(N,M)=M_h^{(0)}\cdot M_v^{(0)}\in\mathbb{R}^{10\times 10}$ is Galois pinching. 
The characteristic polynomial 
  \[
  P(X):=\chi_A(X)=\sum_{i=0}^{10} a_i\,X^i\in\mathbb{Z}[X]
  \]
of the matrix $A=A_6(N,M)\in\mathbb{R}^{10\times 10}$ is monic reciprocal, i.e. $a_{10}=a_0=1$ and $a_i=a_{10-i}$ for $i=1,\dots,5$. Hence there is a cubic polynomial 
\begin{align}\label{Eq:CubicQG=6}
  Q(Y)=Y^5+\sum_{i=0}^4 b_i\,Y^i\in\mathbb{Q}[Y]
\end{align}
such that $1/X^5\cdot P(X)=Q(X+1/X+2)$. The coefficients of $Q(Y)\in\mathbb{Q}[Y]$ are
\begin{align*}
b_4 = & ~a_1-10, \quad b_3 = ~a_2-8\,a_1+35, \quad b_2 =  ~a_3-6\,a_2+20\,a_1-50,\\
b_1 = & ~a_4-4\,a_3+9\,a_2-16\,a_1+25,\quad b_0 =  ~a_5-2\,a_4+2\,a_3-2\,a_2+2\,a_1-2.
\end{align*}
Denote the sextic Weber resolvent of $Q(Y)\in\mathbb{Q}[Y]$ again by $SWR_Q(Y)\in\mathbb{Q}[Y]$. For $m\equiv 2$ modulo $89$ we have that $P(X)=\chi_A(X)$ can be written by irreducible factors modulo $89$ as
\begin{IEEEeqnarray*}{lCl}
    P(X) &\equiv & X^{10} + 4\,X^9 + 63\,X^8 + 33\,X^7 + 39\,X^6 + 71\,X^5 \\
         &       &+ 39\,X^4 + 33\,X^3 + 63\,X^2 + 4\,X + 1~\mbox{modulo $89$}.
\end{IEEEeqnarray*}

Furthermore for $m\equiv 2$ modulo $17$ respectively $m\equiv 2$ modulo $19$ we have that $Q(Y)$ respectively $SWR_Q(Y)$ factorizes as
\begin{align*}\label{Eq:SWR}
Q(Y)     &\equiv Y^5 +4\,Y^4 +Y^2 +6\,Y+16~\mbox{modulo $17$} \qquad \mbox{and}\\
SWR_Q(Y) &\equiv Y^6 +13\,Y^5 +7\,Y^4 +3\,Y^3 +15\,Y^2 +17\,Y+16~\mbox{modulo $19$.}
\end{align*}
Hence $Q(Y)$ is irreducible modulo $17$ and $SWR_Q(Y)$ is irreducible modulo $19$. Thus $P(X)\in\mathbb{Q}[X]$ is irreducible for $m\equiv 2$ modulo $89$ and we identify its Galois group $\text{Gal}(P)$ of $P(X)\in\mathbb{Z}[X]$ again with a subgroup of the hyperoctahedral group $G_5=\mathbb{Z}_2^5\rtimes S_5$.

To ensure that $A=A_6(N,M)$ is Galois pinching we need that $\text{Gal}(P)$ projects surjectively onto $S_5$ or equivalent that $\text{Gal}(Q)=S_5$. With
Theorem \ref{Thm:SexticWR} and Remark \ref{Rem:SWRFullGG} it suffices to find $m\in\mathbb{N}$ such that the discriminant $\text{Disc}(SWR_Q)=\text{Disc}(Q)$ is not a rational square and that the sextic Weber resolvent $SWR_Q(Y)\in\mathbb{Q}[Y]$ does not have a rational root. We have
\begin{displaymath}
\text{Disc}(Q)=c\, f(m)\, (2\,m+3)^{24}
\end{displaymath}
for $c>0$ and an irreducible polynomial $f(m)\in\mathbb{Z}[m]$ of degree $16$. With Siegel's theorem of integral points and the equations above we conclude that $\text{Gal}(Q)=S_5$ for all but finitely many $m\in\mathbb{N}$ with $m\equiv 2 $ modulo $p\in\{17,\,19\}$.

Next we want to restrict $m\in\mathbb{N}$ further such that $\text{Gal}(P)\neq S_5$ and $\text{Gal}(P)\neq H_{5,i}$ for $i=1,2,3$. For the expressions $\Delta_{5,1}=\delta_{5,1}^2$ and $\Delta_{5,2}=\delta_{5,2}^2$ from Lemma \ref{Lem:MaxSubDiscr} and Remark \ref{Rem:SqDelta} we get
\begin{align*}
\Delta_{5,1} & =Q(0)\,Q(4)=c\cdot g(m)\cdot(m-1)(4m+1)(2m+3)^8 \quad\text{and} \\
\Delta_{5,2} & =\text{Disc}(Q) \cdot \Delta_{5,1},
\end{align*}
where $c\in\mathbb{Z}$ and $g(m)\in\mathbb{Z}[m]$ is an irreducible polynomial of degree $10$. This shows that $\text{Gal}(P)\neq H_{5,i}$ for $i=1,2$ and almost all $m\in\mathbb{N}$ with $m\equiv 2$ modulo $p$, where $p\in\{17,\,19,\,89\}$. Furthermore if $m\equiv 2$ modulo $29$ then $P(X)$ can be written in irreducible factors modulo $29$ as 
\begin{IEEEeqnarray*}{lCl}
    P(X) &\equiv& (X+ 15)(X + 2)(X^2 + 7X + 7)(X^2 + X + 25) \\
         &      &(X^4 + 22X^3 + 21X^2 + 22X + 1)~\text{modulo}~29.
\end{IEEEeqnarray*}

Since $29$ does not divide the discriminant $\text{Disc}(P)$ of $P(X)\in\mathbb{Q}[X]$ for $m\equiv 2$ modulo $29$, we conclude with the theorem of Dedekind that in this case $\text{Gal}(P)$ contains a permutation of cycle type $(4,2,2,1,1)$. But $H_{5,3}$ does not contain a permutation of this cycle type as we showed in Appendix \ref{Sect:PermTypes}.

The discussion from above almost showed:

\begin{proposition}
The matrix $A_6(N,M)\in\mathbb{R}^{10\times 10}$ is Galois pinching for all but perhaps finitely many $m\in\mathbb{N}$ with $m\equiv 2$ modulo $p$, where $p\in\{17,\,19,\,29,\,89\}$.
\end{proposition}

\begin{proof}
Let $m\in\mathbb{N}$ such that the quintic polynomial $Q(Y)\in\mathbb{Q}[Y]$ from Equation (\ref{Eq:CubicQG=6}) is irreducible. By substituting $t=Y-b_4/5$ we can bring $Q(Y)$ in depressed form
  \[
  DQ(t)=t^5+p\,t^3+q\,t^2+r\,t+s\in\mathbb{Q}[t].
  \]
By an analysis of the four discriminants $F_i(DQ)$ ($i=1,2,3,4)$ from (\ref{Eq:Dis5RealRoots}), for the polynomial $DQ(t)$ from above, we can easily see that for all $i=1,2,3,4$ and $m$ big enough the inequalty $F_i(DQ)$ holds. From \cite{Hou96} we know that in this case the depressed polynomial $DQ(t)$ and hence $Q(Y)$ has five real roots. Denote the roots of $Q(Y)$ by $\mu_i$ $(i=1,\dots,5)$. We have the equality
  \[
  \mu_i=\lambda_i+\lambda_i^{-1}+2
  \]
for all $i=1,\dots,5$, where $\lambda_i$ and $\lambda_i^{-1}$ are roots of the reciprocal characteristic polynomial $P(X)$ of $A_6(N,M)$. Furthermore for $m$ big enough all the coefficients $b_i$ ($i=1,\dots,5$) of $Q(Y)$ are positive. With Decarte's rule of signs we conclude that in this situation all the roots $\mu_i$ $(i=1,\dots,5)$ are negative real numbers. As before in \ref{Prop:GalPinG4} we know now that the roots $\{\lambda_i\,\lambda_i^{-1}\mid i=1,\dots,5\}$ of $P(X)$ are real. Putting this argument together with the arguments we did before for $\text{Gal}(P)$, we see that $A_6(N,M)$ is Galois pinching for $m\in\mathbb{N}$ big enough such that $m\equiv 2$ modulo $p$, where $p\in\{17,\,19,\,29,\,89\}$.
\end{proof}

Denote by $B_6(N,M)=M_\gamma^{(0)}\in\mathbb{R}^{10\times 10}$ the representation matrix with respect to the basis $B^{(0)}$ of the map $D_\gamma$ on the non-tautological $H_1^{(0)}(\mathcal{O}_{N,M}^{(6)},\mathbb{R})$ from Subsection \ref{SubSec:DTG6}. Then $B_6(N,M)$ is unipotent and the subspace 
  \[
  (B_6(N,M)-\textrm{Id})(\mathbb{R}^{10})
  \]
is one-dimensional and hence not a Lagrangian subspace with respect to $\Omega$. Furthermore $A_6(N,M)$ and $B_6(N,M)$ do not commute but perhaps for finitely many $m\in\mathbb{N}$. Putting thhis together with the previous Proposition about the matrix $A_6(N,M)$, Proposition 4.3 in \cite{matheus15} and Theorem 9.10 in \cite{prasad14} implies:

\begin{theorem}\label{t.A.g6}
The genus six Origamis $\mathcal{O}_{N,M}^{(6)}\in\mathcal{H}(10)$ with $N=3+2m$ and $M=6+4m$ have Kontsevich-Zorich monodromies with finite index in $\text{Sp}((H_1^{(0)}(\mathcal{O}_{N,M}^{(6)},\mathbb{Z}))$ for all but finitely many $m\in\mathbb{N}$ such that $m\equiv 2$ modulo $p$, where $p\in\{17,19,29,89\}$.
\end{theorem}

\newpage
\appendix

\section{Cycle types}\label{Sect:PermTypes}

We call $\varphi_k$ the map that identifies the group $H_{k,3}\leq G_k$ from Proposition \ref{Prop:SubgrPhiOnto} with a subgroup of the permutation group $\text{Sym}(\{\lambda_i,\lambda_i^{-1}\mid i=1,\dots,k\})$. 

We first consider $k=3$ and $\varphi_3\colon H_{3,3}\to \text{Sym}(\{\lambda_i,\lambda_i^{-1}\mid i=1,2,3\})$. We have
\begin{align*}
\varphi_3((-1,-1,-1),(123))=&(\lambda_1\lambda_2^{-1}\lambda_3\lambda_1^{-1}\lambda_2\lambda_3^{-1}), \\
\varphi_3((-1,-1,-1),(1,2))=&(\lambda_1\lambda_2^{-1})(\lambda_2\lambda_1^{-1})(\lambda_3 \lambda_3^{-1}),\\
\varphi_3((+1,+1,+1),(123))=&(\lambda_1\lambda_2 \lambda_3)(\lambda_1^{-1} \lambda_2^{-1} \lambda_3^{-1}) \\
\varphi_3((+1,+1,+1),(1,2))=&(\lambda_1\lambda_2)(\lambda_1^{-1} \lambda_2^{-1})(\lambda_3)(\lambda_3^{-1}).
\end{align*}
The element $(123)\in S_3$ is of cycle type $(3)$ and $(12)\in S_3$ is of type $(2,1)$. These two cycle types are the only non-trivial cycle types that can appear in $S_3$. Hence the calculations from above show that the only non-trivial cycle types that occur in $\varphi_3(H_{3,3})\leq \text{Sym}(\{\lambda_i,\lambda_i^{-1}\mid i=1,2,3\})$ are $(6),(3,3),(2,2,2)$ and $(2,2,1,1)$.

For $k=5$ the permutation group $S_5$ has permutations of type $(5)$, $(4,1)$, $(3,2)$, $(3,1,1)$, $(2,2,1)$, $(2,1,1,1)$ and $(1,1,1,1,1)$. If we want to determine the cycle types of the elements of $H_{5,3}$ as a subgroup of $\text{Sym}(\{\lambda_i,\,\lambda_i^{-1}\mid i=1,\dots,5\})$, then because of symmetry reasons it is sufficient to determine the permutations $\varphi_5((-1,\dots,-1),\sigma))$ and $\varphi_5((+1,\dots,+1),\sigma))$, where $\sigma\in S_5$ is a represent of a cycle type of $S_5$ as above. We have
\begin{align*}
\varphi_5((-1,\dots,-1),(12345))=   &(\lambda_1\lambda_2^{-1}\lambda_3\lambda_4^{-1}\lambda_5\lambda_1^{-1}\lambda_2\lambda_3^{-1}\lambda_4\lambda_5^{-1})\\
\varphi_5((-1,\dots,-1),(1234))=    &(\lambda_1\lambda_2^{-1}\lambda_3\lambda_4^{-1})(\lambda_2\lambda_3^{-1}\lambda_4\lambda_1^{-1})(\lambda_5\lambda_5^{-1})\\
\varphi_5((-1,\dots,-1),(123)(45))= & (\lambda_1\lambda_2^{-1}\lambda_3\lambda_1^{-1}\lambda_2\lambda_3^{-1})(\lambda_4\lambda_5^{-1})(\lambda_5\lambda_4^{-1}) \\
\varphi_5((-1,\dots,-1),(123))=     & (\lambda_1\lambda_2^{-1}\lambda_3\lambda_1^{-1}\lambda_2\lambda_3^{-1})(\lambda_4\lambda_4^{-1})(\lambda_5\lambda_5^{-1}) \\
\varphi_5((-1,\dots,-1),(12)(34))=  & (\lambda_1\lambda_2^{-1})(\lambda_2\lambda_1^{-1})(\lambda_3\lambda_4^{-1})(\lambda_4\lambda_3^{-1})(\lambda_5\lambda_5^{-1})\\
\varphi_5((-1,\dots,-1),(12))=      & (\lambda_1\lambda_2^{-1})(\lambda_2\lambda_1^{-1})(\lambda_3\lambda_3^{-1})(\lambda_4\lambda_4^{-1})(\lambda_5\lambda_5^{-1}).
\end{align*}
We conclude that for the subgroup $\varphi_k(H_{k,3})\leq \text{Sym}(\{\lambda_i,\,\lambda_i^{-1}\mid i=1,\dots,5\})$ only non-trivial permutations of type $(10)$, $(4,4,2)$, $(6,2,2)$, $(2,2,2,2,2)$ and of type $(5,5)$, $(4,4,1,1)$, $(3,3,2,2)$, $(3,3,1,1,1,1)$, $(2,2,2,2,1,1)$ and $(2,2,1,1,1,1,1,1)$ can occur.

\section{Representation matrix for Dehn twist}\label{Sec:DehnTwist}
The following matrix is the representation matrix $M_\alpha^{(0)}$ for the restriction of the Dehn twist $D_\alpha$ in direction $(1,-2)$ to $H_1^{(0)}(\mathcal{O}_{N,M}^{(5)},\mathbb{Q})$ with respect to the basis $B^{(0)}$ from Section \ref{Sec:DehnTwistg=5}. 
\begin{center}
\begin{sideways}
$\begin{pmatrix}
-8-N-4m & -9-N-4m & 0 &  -9-N-4m  & -9-N-4m & -9-N-4m & 0 & -9-N-4m \\
   2    &    3    & 0 &      2    &    2    &    2    & 0 &    2    \\
   2    &    2    & 1 &      2    &    2    &    2    & 0 &    2    \\
   2    &    2    & 0 &      3    &    2    &    2    & 0 &    2    \\ 
  -4    &   -4    & 0 &     -4    &   -3    &   -4    & 0 &   -4    \\
  -4    &   -4    & 0 &     -4    &   -4    &   -3    & 0 &   -4    \\
 7+N+4m & 7+N+4m  & 0 &  7+N+4m   &  7+N+4m & 7+N+4m  & 1 &  7+N+4m \\ 
 - 4    &   -4    & 0 &     -4    &   -4    &   -4    & 0 &   -3     
\end{pmatrix}$
\end{sideways}
\end{center}

\end{document}